\documentclass[11pt,reqno]{amsart}
\usepackage{amssymb}
\usepackage{mathrsfs}
\PassOptionsToPackage{dvipsnames}{xcolor}
\usepackage{pgfplots}
\pgfplotsset{compat=1.11}
\usepackage{amsbsy}
\usepackage{amsfonts}
\usepackage{amsmath}
\usepackage{amsthm}
\usepackage{amssymb}
\usepackage{ytableau}
\usepackage[all,arc]{xy}
\usepackage{tikz-cd}
\usepackage{enumerate}
\usepackage{float}

\newcommand{\SSS}{S}

\newcommand{\NN}{\mathbb{N}}
\newcommand{\ordr}{\operatorname{ord}}

\newcommand{\ddR}{\QQ}
\newcommand{\ddS}{U}
\newcommand{\dda}{d}
\newcommand{\dnu}{\operatorname{ord}_t}
\newcommand{\ddnu}{\alpha}

\newcommand{\DR}{R} 
\newcommand{\DD}{D} 
\newcommand{\FF}{{\mathbb{F}}}
\newcommand{\KK}{\mathbb{K}}
\newcommand{\QQ}{\mathbb{Q}}

\newcommand{\ZZ}{\mathbb{Z}}

\newcommand{\ivr}{\alpha} 

\newcommand{\Gbb}{\overline{\overline{G}}}
\newcommand{\Gb}{{\overline{G}}}
\newcommand{\nub}{\overline{\nu}}

\newcommand{\ba}{{\bf a}}
\newcommand{\bb}{\mathfrak{b}}

\newcommand{\pp}{\mathfrak{p}} 

\newcommand{\tp}{{t}}

\newcommand{\sD}{\mathcal{D}}

\newcommand{\sT}{\mathcal{T}}

\newcommand{\sI}{\mathscr{I}}

\newcommand{\T}{\mathcal{T}} 

\newcommand{\PP}{\mathbb{P}}
\newcommand{\order}{\operatorname{ord}}

\newcommand{\bt}{\beta}
\newcommand{\coloneqq}{:=}
\newcommand{\Prim}{{\rm Prim}}

\theoremstyle{plain}
\newtheorem{theorem}{Theorem}[section]
\newtheorem{prop}[theorem]{Proposition}

\newtheorem{thm}[theorem]{Theorem} 
\newtheorem{lem}[theorem]{Lemma}

\theoremstyle{definition}
\newtheorem{definition}[theorem]{Definition}

\newtheorem{example}[theorem]{Example}

\theoremstyle{remark}
\newtheorem{remark}[theorem]{Remark}
\newtheorem{rem}[theorem]{Remark}

\numberwithin{equation}{section}

\definecolor {UMblue}  {RGB}{0, 39, 76}
\definecolor {UMmaize} {RGB}{255, 203, 5}

\definecolor {color_b}{RGB}{255,0,0}
\definecolor {color_c}{RGB}{20, 200, 30}
\definecolor {color_a}{RGB}{0,0,255}

\definecolor{lgreen} {RGB}{180,210,100}
\definecolor{dblue}  {RGB}{20,66,129}
\definecolor{ddblue} {RGB}{11,36,69}
\definecolor{lred}   {RGB}{220,0,0}
\definecolor{nred}   {RGB}{224,0,0}
\definecolor{norange}{RGB}{230,120,20}
\definecolor{nyellow}{RGB}{255,221,0}
\definecolor{ngreen} {RGB}{98,158,31}
\definecolor{dgreen} {RGB}{78,138,21}
\definecolor{nblue}  {RGB}{28,130,185}
\definecolor{jblue}  {RGB}{20,50,100}

\title[The factorial function and generalizations, extended]{The factorial function and generalizations, extended} 
\author{Jeffrey C. Lagarias}
\address{Department of Mathematics, University of Michigan, 530 Church Street, Ann Arbor, MI 48109--1043, USA}
\email{lagarias@umich.edu} 
\author{Wijit Yangjit} 
\email{yangjit@umich.edu}
\date{March 4, 2024}

\begin{document}

\begin{abstract}
This paper presents an extension of Bhargava's theory of factorials associated to any nonempty subset $\SSS$ of $\ZZ$.
Bhargava's factorials $k!_\SSS$ are invariants, constructed using the notion of $p$-orderings of $\SSS$ where $p$ is a prime. This paper
defines $b$-orderings of any nonempty subset $\SSS$ of $\ZZ$
for all integers $b\ge2$, as well as ``extreme" cases $b=1$ and $b=0$.
It defines generalized factorials 
 $k !_{\SSS,\T}$ and generalized binomial coefficients $\binom{k+\ell}{k}_{\SSS, \sT}$ as nonnegative integers, 
  for all nonempty $\SSS$ and allowing only $b$ in  $\T \subseteq \NN$.
It computes $b$-ordering  invariants  when $\SSS$ is $\mathbb{Z}$ and when $\SSS$ is the set of all primes.
\end{abstract}

\maketitle
\setlength{\baselineskip}{1.0\baselineskip}

%
%

\section{Introduction}

This paper  presents an extension  of Bhargava's theory of  generalized factorials associated 
to any nonempty subset $\SSS$ of the integers $\ZZ$ as presented in \cite{Bhar:00}.
The  factorials of Bhargava are based on  invariants called $p$-sequences of $\SSS$,
indexed by the nonnegative integers $k$, which are defined for all primes $p$ of $\ZZ$,
and are denoted $\nu_k(\SSS,p)$.
They are defined using  a notion of $p$-ordering of the elements in $\SSS$, and each invariant $\nu_k(\SSS,p)$ is a nonnegative power of $p$. 
 Bhargava's  generalized
factorials $k!_{\SSS}$ are produced
in factored form as the product over all primes $p$ of the $k$-th invariants $\nu_k(\SSS,p)$;
these invariants are $1$ for all $p >k$, so the factorization is a finite product.

This paper defines  $b$-orderings and $b$-sequences
of $\SSS$  for all integers $b \ge 2$, which coincide with Bhargava's  definitions when $b$ is  a prime.
We also define these notions for  the extreme cases $b=1$ and $b=0$. The main result is the
well-definedness of $b$-sequences independent of the choice of $b$-ordering.

 The existence  of $b$-sequences  of $\SSS$  for all bases $b$ inserts an extra degree of freedom in  Bhargava's theory. 
 Using them we define for  a given nonempty subset $\SSS$ of $\ZZ$ and a given subset $\sT$ of allowed $b$-values,
 taken from $\NN=\{0, 1, 2, 3, 4,\dots\}$,
  generalized factorials $k!_{S, \sT}$ and generalized binomial
 coefficients $\binom{k+\ell}{k}_{\SSS, \sT}$, which are always  nonnegative integers .
 The case where $\sT= \PP$, the set of all primes, recovers Bhargava's generalized factorials.

We review   Bhargava's theory for $\ZZ$ in Section \ref{sec:1a}.
The definition of $b$-orderings  for general $b$  is 
introduced  in  Section \ref{sec:2} as Definition \ref{def:b-orderings}. The main results are 
 formulated  in Section \ref{sec:2}, in parallel to Bhargava's results; 
 they are  proved in Sections \ref{sec:Bhargava-Z[[t]]}  through \ref{sec:7a}.  The well-definedness result is proved via
 suitable Bhargava $\pp$-ordering invariants for the prime ideal $\pp=t \QQ[[t]]$ in the formal power series ring $\QQ[[t]]$;
 we give a new max-min characterization of such invariants in Section \ref{sec:Bhargava-Z[[t]]}.  
 We carry $b$-ordering invariants into $\pp$-ordering invariants using congruence-preserving mappings,
 presented in Section \ref{sec:5}. 
 We give formulas for  the  $b$-ordering invariants for 
 $\SSS= \ZZ$ and  $\SSS=\PP$, in Sections 7 and \ref{sec:7a}. 

%
%

\section{Bhargava's $p$-orderings and generalized factorials}\label{sec:1a}

In  1996 Bhargava \cite{Bhar:96} presented  a theory of generalized factorials for a 
large class of commutative rings $\DR$
 called Dedekind-type rings,  which include all Dedekind domains.  
   Bhargava's generalized factorials are associated to a nonempty set $S$ of elements of $\DR$ and to
    the set $\PP(R)$
      of all  nonzero prime ideals  of $\DR$.
For Dedekind rings $\DR$, Bhargava's generalized factorials $k!_S$ are defined to be ideals in $\DR$.
For some history of his work see   Cahen, Chabert and Kedlaya \cite{CCK:17}.

Bhargava originally developed his construction of generalized factorials for  $\mathbb{Z}$.
In  2000 he described the $\ZZ$-theory in detail (\cite{Bhar:00}). In the $\ZZ$-theory 
generalized factorials and binomial coefficients  are treated as
elements of the nonnegative integers $\NN$.

Bhargava's definitions of factorials and binomial coefficients associated
to a nonempty subset $\SSS$ of $\ZZ$ are given via their  prime factorizations.
For factorials a factorization formula, sometimes called de Polignac's formula, goes back to 
Legendre \cite[p.8]{Legendre:1808}:
\begin{equation}\label{eqn:de-Polignac}
k!= \prod_{p \in \PP} p^{\left \lfloor \frac{k}{p}\right \rfloor +\left \lfloor \frac{k}{p^2} \right\rfloor + \left\lfloor \frac{k}{p^3} \right\rfloor + \cdots},
\end{equation}
where $\PP= \{ 2,3, 5, \dots\}$ denotes the set of all prime numbers. 
This factorization  for a single $p$ in the factorial are given by the $k$-th term  in a $p$-sequence
 for $\SSS= \ZZ$, based on his  notion of $p$-orderings of $\SSS$, cf. \eqref{eqn:classic-factorial}.
 The study of the prime factorizations of binomial coefficients traces back to work of Kummer \cite{Kum:1852} in 1852;
see also Granville \cite{Gra:97}.

%
%

\subsection{Bhargava's $p$-orderings of nonempty subsets $\SSS$ of $\ZZ$}\label{subsec:13a}

\begin{definition} \label{def:p-ordering}
Let $p$ be a prime. Let $S$ be a nonempty subset of $\mathbb{Z}$.
 A \emph{$p$-ordering of $S$} is any sequence $\mathbf{a}=\left(a_i\right)_{i=0}^\infty$ of elements of $S$ that can be formed recursively as follows:
\begin{enumerate}
\item[(i)] $a_0\in S$ is chosen arbitrarily;
\item[(ii)] Given $a_j\in S$, $j=0,\dots,i-1$, the next element $a_i\in S$ is chosen so that $a=a_i$
minimizes 
the highest power of $p$ dividing the product $\prod_{j=0}^{i-1}\left(a-a_j\right)$.
\end{enumerate}
\end{definition}

Bhargava defines  $\nu_i(S,p,\mathbf{a})$ 
to be the largest power of $p$ that divides $\prod_{j=0}^{i-1}\left(a_i-a_j\right)$. 
Bhargava calls the sequence 
 $\left(\nu_i(\SSS,p,\mathbf{a})\right)_{i=0}^\infty$ the \emph{$p$-sequence of $S$ associated to the $p$-ordering $\mathbf{a}$}.
Bhargava \cite[Theorem~5]{Bhar:00} showed:

%
%

\begin{thm}[Bhargava \cite{Bhar:00}]\label{thm:bhar1}
For a nonempty subset $\SSS$ of  $\ZZ$ and a prime  $p$, the associated $p$-sequence
of a $p$-ordering of $\SSS$ is independent of the choice of $p$-ordering. 
\end{thm}

By Theorem \ref{thm:bhar1} one may set $\nu_i(\SSS,p) :=\nu_i(\SSS,p,\mathbf{a})$ for any
choice of  $p$-ordering of $\SSS$,
and call $(\nu_i (\SSS, p))_{i=0}^{\infty}$ the  {\em $p$-sequence} of $\SSS$.
By convention $\nu_0(\SSS, p) = 1$.

We may write
\begin{equation}\label{eqn:22}
\nu_i(\SSS, p,\mathbf{a})=p^{\ivr_i^{\ast}(S,p,\mathbf{a})},
\end{equation}
where
\begin{equation}\label{eqn:23}
\ivr_i^{\ast}(\SSS,p,\mathbf{a}):=\order_p\bigg(\prod_{j=0}^{i-1}\left(a_i-a_j\right)\bigg).
\end{equation}
Here  $\order_p(\cdot): \ZZ \to \NN \cup \{ +\infty\}$ denotes the additive $p$-adic valuation given by
\begin{equation}
\order_p(k):=\sup\left\{\alpha\in\mathbb{N}:p^\alpha\text{ divides }k\right\},
\end{equation}
with $\order_p(0)= +\infty$. By the convention above, 
$\ivr_0^{\ast}(\SSS, p, \mathbf{a})=0$.

Theorem \ref{thm:bhar1} implies  that one may write  
 $ \ivr_i^{\ast}(\SSS, p):= \ivr_i^{\ast}(\SSS, p, \mathbf{a})$  as invariants  of $\SSS$,  independent of
the choice of $p$-ordering $\mathbf{a}$.
We call the data $\left( \ivr_i^{\ast}(\SSS,p) \right)_{i =0}^{\infty}$ the {\em  $p$-exponent sequence of $\SSS$};
 it  is equivalent data to 
 $(\nu_i(\SSS, p))_{i=0}^{\infty}$, provided $p$ is known.

%
%

\subsection{Bhargava's generalized factorials and generalized binomial coefficients associated to nonempty subsets $\SSS$ of $\ZZ$}\label{subsec:14a}

Bhargava used $p$-sequences  to define his generalized factorials. 
The \emph{generalized factorial function associated to a nonempty subset $S$ of $\ZZ$}, denoted $k!_S$, is defined by
\begin{equation}
k!_S:=\prod_{p\in \PP} \nu_k(S,p) = \prod_{p \in \PP} p^{\ivr_k^\ast(S,p)}.
\end{equation}
This definition constructs factorials  expressed
multiplicatively  in terms of  their prime factorizations.

In the special case $S=\mathbb{Z}$, Bhargava showed that the generalized factorials agree with the usual factorials. To do this, 
Bhargava \cite[Proposition~6]{Bhar:00} showed

%
%
\begin{thm}[Bhargava \cite{Bhar:00}]\label{thm:Bhargava-simultaneous-p-ordering}
The natural ordering $0,1,2,\dots$ of the nonnegative integers
 forms a $p$-ordering of $S= \mathbb{Z}$ simultaneously for all primes $p$.
\end{thm}

From Theorem \ref{thm:Bhargava-simultaneous-p-ordering}, Bhargava deduces that
\begin{equation}\label{eqn:classic-factorial}
\nu_k(\mathbb{Z},p)=w_p\bigg(\prod_{j=0}^{k-1}(k-j)\bigg)=w_p(k!),
\end{equation}
where $w_p(a)$ denotes the highest power of $p$ dividing $a$ (i.e., $w_p(a)=p^{\order_p(a)}$).

Therefore for $S= \mathbb{Z}$ one has
\begin{equation}
k!_\mathbb{Z}=\prod_{p\in\mathbb{P}}w_p(k!)
= \prod_{p \in \PP} p^{\order_p(k!)}
=k!.
\end{equation}

One can in many cases compute generalized factorials by directly determining $p$-sequences.
For the set $S = \PP$, Bhargava \cite[Example 21]{Bhar:00}
stated the following result.

%
%
\begin{thm}[Bhargava \cite{Bhar:00}]\label{thm:Bhargava-prime}
For the set $S= \PP$ of all primes in $\ZZ$,  we have
\begin{equation}
k!_{\PP} = \prod_{\substack{p\in\mathbb{P}\\p \le k}} p^{\left\lfloor \frac{k-1}{p-1}\right\rfloor +\left\lfloor \frac{k-1}{p(p-1)}\right\rfloor
 + \left\lfloor \frac{k-1}{p^2(p-1)}\right\rfloor+\cdots}.
\end{equation}
\end{thm}

Bhargava notes that $k!_{\PP}$  coincides with  $2^{\left\lfloor \frac{k}{2}\right \rfloor}$ times the product of the
denominators of the first $\left\lfloor \frac{k}{2} \right\rfloor$ Bernoulli numbers.  This sequence appears in
study of the ring of univariate polynomials over $\QQ$ that are integer-valued on the set of prime numbers,
see  Chabert et al. \cite{CCS:97} and Chabert \cite{Ch:96} and \cite{Ch:07}. 
Other  interesting examples, including the sets $S= \{ n^k: n \ge 1\}$, 
appear in Fares and Johnson \cite{FJ:12} and, for twin primes, in \cite{AC:15}.

Bhargava also considered generalized binomial coefficients. One can define
$$
\binom{k+\ell}{k}_S:=\frac{(k+\ell)!_S}{k!_S\ell!_S}
$$
as  a rational number (or $+\infty$ if
a zero denominator is present).
Bhargava \cite[Theorem~8]{Bhar:00} showed:
%
%
\begin{thm}[Bhargava \cite{Bhar:00}]\label{thm:bhar4}
For any nonempty subset $\SSS$ of $\ZZ$, and any nonnegative integers $k$ and $\ell$, 
the generalized factorial $(k+\ell)!_S$ is divisible by  $k!_S\ell!_S$.
\end{thm}

In other words,  generalized binomial coefficients
$\binom{k+\ell}{k}_S := \frac{(k+\ell)!_S}{k!_S\ell!_S}$
are always integers. However, such binomial
 coefficients generally do not satisfy  a Pascal-type recurrence.

Bhargava \cite[Theorem ~10]{Bhar:00} proved a divisibility result for general products
of pairwise differences of a sequence of $n+1$ integers in $\SSS$.

%
%
\begin{thm}[Bhargava \cite{Bhar:00}]\label{thm:bhar10}
Let $a_0, a_1, \cdots, a_n \in \SSS$ be any $n+1$ integers. Then the product
\begin{equation*} 
\prod_{0 \le i < j\le n} (a_i-a_j)
\end{equation*} 
is a multiple of $0!_{\SSS} 1!_{\SSS} \cdots n!_{\SSS}$. (This is sharp!)
\end{thm}

Here ``This is sharp!" means that the constant $0!_{\SSS} 1!_{\SSS} \cdots n!_{\SSS}$
cannot be improved, taken over  all sequences of length $n+1$ drawn from $\SSS$.

%
%

\section{Main results }\label{sec:2}

We define  $b$-orderings, $b$-sequences and $b$-exponent sequences for nonempty
subsets $\SSS$ of $\ZZ$, 
for all integers $b \ge 2$, and
we also define these quantities  for  $b=0$ and $b=1$. We define generalized factorials
 $k!_{S, \sT}$ allowing $b$ in any set $\sT \subseteq \NN$.

%
%

\subsection{ $b$-orderings  and $b$-exponent sequences for arbitrary bases $b \ge 2$ in $\ZZ$.}\label{sec:12}

 It is convenient to  give a definition of $b$-exponent sequences
 that applies   in a more general context, that of $\SSS$-test sequences.

%
%
\begin{definition}\label{def:test-seq}
 An {\em $\SSS$-test sequence} 
 $\mathbf{a}=\left(a_i\right)_{i=0}^\infty$ is any infinite sequence of elements drawn from $\SSS$
 (allowing repetitions of elements).
For integer $b \ge 2$, the   \emph{ $b$-exponent sequence of the $\SSS$-test sequence $\mathbf{a}$}, denoted 
$\left(\ivr_i(\SSS,b,\mathbf{a})\right)_{i=0}^\infty$, is given by
\begin{equation}\label{def:b-sequence-wrt-a}
\ivr_i(\SSS,b,\mathbf{a}):=\sum_{j=0}^{i-1}\ordr_{b}\left(a_i-a_j\right),
\end{equation}
where $\ordr_{b}(a)$ for $b\ge2$ is defined for $a \in \ZZ$ by
\begin{equation}\label{eqn:ord-b}
\ordr_{b}(a):=\sup \{k \in \NN \, : \,  b^k  \mid a \},
\end{equation}
which makes $\ordr_{b}(0)= +\infty$ for all $ b \ge 2$,.
By convention $\ivr_0(\SSS, b, \mathbf{a}) =0,$ for all $b \ge 2$.
 The   {\em $b$-sequence} of the $\SSS$-test $\mathbf{a}$ is then
$( b^{\ivr_i(\SSS, \mathbf{a})})_{i=0}^{\infty}$.
\end{definition}

We note that:
\begin{enumerate}
\item[(1)] $\ivr_i(\SSS,b,\mathbf{a})\in\mathbb{N}\cup\{+\infty\}$.
\item[(2)] If $\SSS$ is nonempty and finite, then $\ivr_i(S, b,\mathbf{a})<+\infty$ for at most $\vert S\vert$ values of $i$.
\end{enumerate}

We define  $b$-orderings of $\SSS$ to be  a restricted class of $\SSS$-test  sequences, determined by
a greedy algorithm that minimizes  $\sum_{j=0}^{i-1} \ordr_b(a_i - a_j)$ at each step.

%
%
\begin{definition}\label{def:b-orderings}
 Let $S$ be a nonempty subset of  $\ZZ$, and let $b \ge 2$ be an integer.
  We call an $\SSS$-test sequence $\mathbf{a}=\left(a_i\right)_{i=0}^\infty$  a  
 \emph{$b$-ordering of $\SSS$} if for all $i=1,2,3,\dots$, the choice of $a_i$
 attains a minimum
\begin{equation}\label{eqn:b-orderings}
\sum_{j=0}^{i-1}\ordr_{b}\left(a_i-a_j\right)=\min_{a'\in \SSS}\,\bigg( \sum_{j=0}^{i-1}\ordr_{b}(a'-a_j) \bigg).
\end{equation}
\end{definition}

Given  any initial value $a_0\in \SSS$, one can find a $b$-ordering having that initial value using the recurrence \eqref{eqn:b-orderings}. There will be more than one 
 $b$-ordering of $\SSS$, unless $\SSS$ is a singleton.

The definition of $b$-orderings differs in form  Definition \ref{def:p-ordering}  for $p$-ordering,
given in \cite{Bhar:00},  see  \eqref{eqn:23}. 
When $b=p$ is a prime,
the definitions agree,  due
to the  valuation identity
\begin{equation*}
\sum_{j=0}^{i-1} \ordr_{p} (a'- a_j) = \ordr_{p} \bigg( \prod_{j=0}^{i-1}(a'- a_j)\bigg).
\end{equation*}
For composite $b$, the function $\ordr_b(\cdot)$
is not an additive valuation, and in general  one has only
\begin{equation*}
\ordr_b(a_1 a_2) \ge \ordr_b(a_1) + \ordr_b(a_2),  \quad \mbox{for}  \quad a_1, a_2 \in \ZZ.
\end{equation*}
In 2009  Bhargava \cite{Bhar:09} introduced the additive definition \eqref{eqn:b-orderings} in the case  $b=p$, as a  special case
of his generalized $p$-orderings, see Section \ref{subsec:91}.

 This paper's focus is  $b \ge 2$, but
Definitions \ref{def:test-seq} and \ref{def:b-orderings} are extendable  to 
``extreme" cases $b=0$ and $b=1$.
 These cases are discussed  in Section \ref{subsec:proofs-53},
for completeness.  The extensions use 
a modified definition of $\ordr_b(\cdot)$, given in \eqref{eqn:ord-b2}.

The main result of this  paper is as follows.

%
%
\begin{theorem}[Well-definedness of the $b$-exponent sequence of $S$ in $\ZZ$]\label{thm:well-definedness}
For all  $b \ge 0$ and 
all nonempty subsets $\SSS$  of  $\ZZ$, if  $\mathbf{a}_1$ and $\mathbf{a}_2$ are 
 $b$-orderings of $S$,
 then
 \begin{equation}\label{eqn:b-sequence} 
 \ivr_i\left(\SSS,b,\mathbf{a}_1\right)=\ivr_i\left(S,b,\mathbf{a}_2\right) \quad \mbox{ for all} \quad  i=0,1,2,\dots.
 \end{equation}
\end{theorem}

This generalization to $b$-exponent sequences 
gives a set of well-defined invariants $\ivr_k(S, b) := \ivr_k(S, b, \mathbf{a})$  
for any choice of  $b$-ordering $\mathbf{a}$.
 We call   the sequence $(\ivr_i(S, b))_{i=0}^\infty$
 the {\em $b$-exponent sequence of $\SSS$}, and we call $(b^{\ivr_i(\SSS, b)})_{i=0}^{\infty}$ the {\em $b$-sequence of $\SSS$}.
 
 The cases $b=0$ and $b=1$ included in Theorem \ref{thm:well-definedness} 
 are ``extreme" in the sense that for each $b \ge 2$ and all 
 nonempty $\SSS$, the bounds
  \begin{equation}\label{eq:extreme}
\ivr_{i}(\SSS, 0) \le \ivr_{i}(\SSS, b) \le \ivr_{i}(\SSS, 1)
 \end{equation}
hold  for all $i \ge 0,$ see Remark \ref{rem:55}.

Individual $b$-orderings $\mathbf{a}$  satisfy the inequalities 
\begin{equation}\label{eq:add-mult}
\ivr_k(S, b, \mathbf{a}) = \sum_{j=0}^{k-1} \ordr_{b} (a_k - a_j) \le 
\ordr_{b} \bigg( \prod_{j=0}^{k-1} (a_k - a_j)\bigg)=: \ivr_k^{\ast}(S, b, \mathbf{a}).
\end{equation} 
When $b=p$ is   prime  \eqref{eq:add-mult} is always  equality.
When $b$ is not a prime strict inequality may
occur in  \eqref{eq:add-mult}.

%
%
\begin{example}\label{exam:35}
The well-definedness result of $b$-exponent sequences does not hold for an analogous
notion  of $b^{\ast}$-orderings based on Definition \ref{def:p-ordering}.
Define a {\em $b^{\ast}$-ordering} to be one that minimizes
$\ordr_b \left( \prod_{j=0}^{k-1} (a' - a_j)\right)$ at each step.  
 Take $b=6$
and $S= \{0, 1, 2, 3, 4, 5\}$.
 Then $\ba =  (a_0, a_1, a_2, a_3) = (0, 1, 2, 5)$ is an initial  $6$-ordering 
 that is also a $6^{\ast}$-ordering, and it has 
$\ivr_3(S, 6, \mathbf{a}) =0$ while  $\ivr_3^{\ast}(S, 6,  \ba) = 1$. 
Here the inequality \eqref{eq:add-mult}  is  strict.
Additionally  $\ba^{'}= (a_0^{'}, a_1^{'}, a_2^{'}, a_3^{'}) = (0, 2, 4, 5)$ is  
an  initial $6^{\ast}$-ordering, and it has
$ \ivr_3^{\ast}(S, 6, \ba')=0$,
which shows that  the value of $\ivr_3^{\ast}(S, 6, \ba)$ is dependent on the choice of 
 $6^{\ast}$-ordering $\mathbf{a}$ of $S$. 
\end{example}

  We prove     Theorem \ref{thm:well-definedness}, as part of  Theorem \ref{thm:main-propertyCZ},
  in Section \ref{subsec:52}. 
    The  method of proof is a reduction to a special case of Bhargava's main theorem:
  the well-definedness of $t$-sequences  associated to $t$-orderings 
 for  the prime ideal $ (t) = t \ddR[[t]]$ of the
 formal power-series ring $\ddR[[t]]$,  over the rational field $\QQ$, for
 all nonempty subsets $U$ of $\ddR[[t]]$. 
  The ring $\ddR[[t]]$ is  a discrete valuation ring, hence a Dedekind domain. 
  This special case is  stated as Theorem \ref{thm:invariant} in Section \ref{sec:Bhargava-Z[[t]]}, and  
   the paper gives an  independent  proof of this special case, making the paper self-contained.
   
  The  idea  of the reduction uses for each $b$ an injective 
 ``congruence-preserving"  map $\varphi_{b}: \ZZ \to \QQ[[t]]$ that converts $b$-orderings to $t$-orderings for all sets $\SSS \subseteq \ZZ$ simultaneously, and preserves the invariants.
 In general these  maps do not preserve ring structure;
  they may  not  be group  homomorphisms for  ring addition nor
 monoid homomorphisms for  ring multiplication.

%
%

\subsection{Max-min criterion for $b$-exponent sequence invariants and  majorization }\label{subsec:32a}

This    paper    formulates in Theorem \ref{thm:max-min} a new characterization of the $t$-exponent sequence invariants for  Bhargava's
 theory for the ideal $\pp= t \ddR[[t]]$ of the ring $\ddR[[t]]$, 
 in terms of a max-min condition that is visibly independent of the $t$-ordering.
 This characterization   implies  Theorem \ref{thm:invariant}, 
 and yields  new  information about $b$-exponent sequences,   given below. 
 This max-min characterization and its proof are due to the second author 
   (\cite[Proposition 3.4.8]{Yangjit:22}).

For $U $ a subset of $\QQ[[t]]$ the  max-min characterization  yields  restrictions on 
 allowed $t$-exponent sequences for $U$-test sequences,  which in turn 
yields restrictions  on the allowed $b$-exponent sequences for general $\SSS$-test  sequences.

%
%
\begin{thm}[Weak majorization for $\SSS$-test sequences] \label{thm:majorization-b}
Let $b \ge 2$.

 \emph{(1)} For  any subset $\SSS$ of $\ZZ$ with $|\SSS| \ge n$  (allowing $n = +\infty$) and any
 $\SSS$-test sequence   $\mathbf{s}=(s_i)_{i=0}^n$, 
 the associated $b$-exponent sequence values $(\alpha_{k}(\SSS, b, \mathbf{s}))_{k=0}^n$ 
weakly majorize the $b$-exponent sequence invariants $(\ivr_k(\SSS, b))_{k=0}^n$ of $\SSS$. That is,
for $1 \le m < n$, 
\begin{equation}\label{eqn:majorization-1}
\sum_{k=1}^{m} \alpha_{k}(\SSS, b, \mathbf{s})   
\ge\sum_{k=1}^m\ddnu_k\left(\SSS,b \right).
\end{equation}

\emph{(2)} Given a fixed $N \ge 1$,  if  the $\SSS$-test sequence
$(s_i)_{i=0}^N$ is the
initial part of a $b$-ordering of $\SSS$,  then  equality 
holds in \eqref{eqn:majorization-1} for $1 \le m \le N$.
\end{thm}

The max-min characterization   implies monotonicity  
 on the behavior  of  $b$-exponent sequences for sets $\SSS$, 
 under set inclusion, as the set $\SSS$ varies,
 see  Proposition \ref{prop:factorial-ordering} (2).

%
%

\subsection{Generalized factorials from $b$-exponent sequences in $\ZZ$.}

We  use $b$-exponent sequences to  define generalized factorials depending
on two parameters: a subset  $\SSS$ of $\T$ and an allowed subset of ideals $\T\subseteq \NN$.
We state several results about such factorials.

%
%
\begin{definition}\label{definition:factorials-S-T}
 Let $S$ be a nonempty subset of $\ZZ$. Let $\T\subseteq \NN$.
 For $k=0,1,2,\dots$, the
 \emph{$k$-th generalized factorial  associated to $S$ and $\T$}, denoted $k!_{S,\T}$, 
 is  a nonnegative  integer  defined by
\begin{equation}\label{def:factorials-S-T}
k!_{S,\T}:=\prod_{b \in\T}b^{\ivr_k(S,b)}.
\end{equation}
We adopt  the convention $b^{+\infty}=0$ for  all $b \ge 2$ and  $b=0$, while $1^{+\infty}=1$. Also $b^0=1$ for all $b\in\NN$.
(This convention is natural if viewing $b$ as an ideal.)
\end{definition}

We note that:
\begin{enumerate}
\item[(1)] If $\T=\varnothing$, then the product on the right side of \eqref{def:factorials-S-T} is empty; so $k!_{S,\varnothing}=1$.
\item[(2)] If $\T\neq\varnothing$, then $k!_{\SSS,\T} \ge 1$
 if $k<\vert \SSS\vert$. 
 \end{enumerate}
 The cases $k=0$ and $k=1$ are exceptional.
 \begin{enumerate} 
\item[(3)]  
 $0!_{\SSS,\T}=1$, by convention, since the sum  \eqref{def:b-sequence-wrt-a} is empty.
(This definition is consistent with our conventions on $\ivr_0(\SSS, b)$ for $b=0,1$; see  Section \ref{subsec:proofs-53}.) 

\item[(4)] $1!_{\SSS,\T}=1$ holds 
for all $\sT$ whenever the difference set  $s_i-s_j$ ($s_i, s_j \in \SSS$) has greatest common divisor $1$. 
(It is  possible for $1!_{S,\T}>1$  to occur otherwise.) 
\end{enumerate}

In the special case $\SSS=\ZZ$ we can again compute factorials via 
simultaneous $b$-orderings for $\ZZ$, now allowing  all bases  $b \ge 2$. 

%
%
\begin{thm}[Simultaneous $b$-ordering]\label{thm:simultaneous-b-ordering}
The natural ordering $0,1,2,\dots$ of the nonnegative integers $\NN$ 
 forms a $b$-ordering of $S= \mathbb{Z}$ simultaneously for all integers $b \ge 2$.
\end{thm}

This result is proved as Theorem \ref{thm:simultaneous-b-ordering1}.\\

We have  now an extra degree of freedom to vary $\T \subseteq \NN$.

%
%
\begin{example}
The choice  $\T=\PP$ 
recovers  Bhargava's generalized factorials for all $ \SSS \subseteq \ZZ$ see \cite{Bhar:00}, i.e.
$$
k!_{\SSS,\PP}=k!_{S}.
$$
For $\SSS=\ZZ$ this formula  recovers  the usual factorial function $k!_{\ZZ, \PP} = k!_{\ZZ}= k!$.
The factorization  formula \eqref{def:factorials-S-T} then encodes de Polignac's formula for prime  factorization of $k!$. 
\end{example}

For the special case $\SSS=\PP$, the set of  primes,  we obtain 
a generalization of Theorem \ref{thm:Bhargava-prime}.

%
%
\begin{thm}\label{thm:B-prime0}
For the set $S= \PP$ of all primes,  and for all $\sT \subseteq \NN\backslash\{0\}$,
\begin{equation}
k!_{\PP, \sT} = \prod_{\substack{b \in \sT\\\varphi(b)+\omega(b)\le k}}
b^{\left\lfloor \frac{k-\omega(b)}{\varphi(b)}\right\rfloor +\left\lfloor \frac{k-\omega(b)}{b \varphi(b)}\right\rfloor + \left\lfloor \frac{k-\omega(b)}{b^2 \varphi(b)}\right\rfloor+\cdots},
\end{equation}
where $\varphi(b)$ is the Euler totient function and  $\omega(b)$ denotes the number of distinct prime divisors of $b$.
\end{thm}

The product over $\sT$ is finite, since $\varphi(b) \to \infty$ as $b \to \infty.$ The case $\sT = \PP$ corresponds
to Bhargava's result.  Theorem \ref{thm:B-prime0} is proved in Section \ref{sec:7a}.

The next result shows that generalized factorials  satisfy divisibility relations
 when varying either  parameter  $S$ or $\T$ by inclusion. The divisibility relation varying $S$ seems 
 more subtle than that of varying $\sT$. We do not know of a proof based on the direct definition of $b$-ordering;
 instead the proof of Proposition \ref{prop:factorial-ordering} (2) appeals to Proposition \ref{prop:4-prop}, which in turn
 follows using Theorem \ref{thm:subset-decreasing}, proved using the max-min characterization given in Theorem \ref{thm:max-min}. 
%
%
\begin{prop}[Generalized factorial inclusion and divisibility properties]\label{prop:factorial-ordering}

\quad\quad
\begin{enumerate}

\item[\emph{(1)}] Let $S$ be a fixed nonempty subset of  $\ZZ$ and let $\T_1\subseteq\T_2\subseteq \NN$.
Then for integers $0\le k<\vert S\vert$,
$$
k!_{S,\T_1}\quad\text{divides}\quad k!_{S,\T_2}.
$$

\item[\emph{(2)}] Let $S_1\subseteq S_2$ be nonempty subsets of  $\ZZ$ and  let $\T\subseteq \NN$
 be fixed. Then for integers $0\le k<\left\vert S_1\right\vert$,
$$
k!_{S_2,\T}\quad\text{divides}\quad k!_{S_1,\T}.
$$

\end{enumerate}
\end{prop}

We prove this result
 in Section \ref{subsec:61}. The case (2) of Proposition \ref{prop:factorial-ordering} above
 includes Lemma 13 of \cite{Bhar:00} as the special case $\sT= \PP$.

Of particular interest is the special case $\SSS=\ZZ$, and $\sT= \NN$
which gives,
according to (1) of Proposition \ref{prop:factorial-ordering} above,    the generalized factorial 
which has the minimal size for each $k$ over all possible $\SSS$, holding $\T= \NN$
fixed. It is
$$
k!_{\ZZ, \NN}= \prod_{b =2}^\infty b^{\ivr_k( \ZZ, b)}.
$$
We give formulas for  these generalized factorials  in Section \ref{subsec:72}, and give tables 
of their values in Appendix \ref{appendix:tables}.

%
%

\subsection{Generalized-integers  and generalized binomial coefficients }\label{subsec:34}

%
%
\begin{definition}\label{definition:integers-S-T}
Let $S$ be a nonempty subset of  $\ZZ$.
Let $\T \subseteq \NN$.
 For positive integers $n<\vert S\vert$, 
the \emph{$n$th  generalized-integer  ideal associated to $S$ and $\T$}, denoted $[n]_{S,\T}$, is 
the rational number  defined by
\begin{equation}\label{def:nST}
[n]_{S,\T}:=\frac{n!_{S,\T}}{(n-1)!_{S,\T}}.
\end{equation}
We set $[n]_{S,\T}:=0$ for $n \ge \vert S \vert$ (excluding $\T = \{1\}$).
\end{definition}
Since $0!_{S, \T} = 1$, we have $ [1]_{S,\T}= 1!_{S,\T}$.
 The definition \eqref{def:nST} then  yields 
\begin{equation}
n!_{S,\T} = \prod_{j=1}^n [j]_{S,\T}
\end{equation} 
by induction on $n$.

This  definition \eqref{def:nST}  shows that  $[n]_{S, \T}$  is  a  nonnegative rational number.
The next result asserts that $[n]_{S, \T}$  is an integer,
justifying its name ``generalized-integer".


\begin{thm}\label{thm:integer-integer}
Let $S$ be a  nonempty subset of 
$\ZZ$. 
Let $\T\subseteq \NN$.
Then for positive integers $n<\vert S\vert$, the generalized-integer
  $[n]_{S,\T}$ 
is a positive  integer.
\end{thm}

We prove Theorem \ref{thm:integer-integer} in Subsection \ref{subsec:61}. 

%
%
\begin{definition}\label{definition:binomial-S-T}
Let $S$ be a nonempty subset of $\ZZ$. Let $\T\subseteq \NN$.
For integers $0\le\ell\le k<\vert S\vert$, the \emph{generalized binomial coefficient  associated to $S$ and $\T$}, 
denoted $\binom{k}{\ell}_{S,\T}$, is defined by
\begin{equation}
\binom{k}{\ell}_{S,\T}:=\frac{k!_{S,\T}}{\ell!_{S,\T}(k-\ell)!_{S,\T}}.
\end{equation}
\end{definition}

By definition the set $\binom{k}{\ell}_{S,\T}$ is  a positive rational number, or $+\infty$. 

%
%
\begin{thm}\label{thm:binomial-integer}
Let $S$ be a nonempty subset of $\ZZ$ . Let $\T\subseteq \NN$.
Then for integers $0\le\ell\le k<\vert S\vert$, the generalized binomial coefficient $\binom{k}{\ell}_{S,\T}$ is a positive integer.
\end{thm}

We  prove Theorem \ref{thm:binomial-integer} in Subsection \ref{subsec:61}. \\

A final  result is a generalization of 
Theorem \ref{thm:bhar10}.

%
%
\begin{thm}\label{thm:bhar10revisited}
Let $\SSS \subseteq \ZZ$ be given, and let  $\mathbf{a}= (a_i)_{i=0}^{n}$ be any $n+1$ integers in $\SSS$. Then for
any $\sT \subseteq \NN$
 the product
\begin{equation*} 
\prod_{b \in \sT}b^{ \gamma(\SSS, b, \mathbf{a})}
\end{equation*} 
with
\[
\gamma(\SSS, b, \mathbf{a}) := \sum_{0 \le i < j\le n} \ordr_b(a_i-a_j)
\]
is a multiple of $0!_{\SSS, \T} 1!_{\SSS,\T} \cdots n!_{\SSS,\T}$. 
\end{thm}

Theorem \ref{thm:bhar10revisited}  follows by  direct application of  Theorem \ref{thm:majorization-b} (1)  to the
finite  $\SSS$-test sequence $\mathbf{a}$, noting that
$\gamma(\SSS, b, \mathbf{a}) =  \sum_{i=0}^n \ivr_i(\SSS, b, \mathbf{a}).$
We do not know whether Theorem \ref{thm:bhar10revisited} is sharp.

%
%

\subsection{Contents of the paper}\label{sec:35a}

 In Section \ref{sec:Bhargava-Z[[t]]} we prove the max-min characterization
 of $t$-ordering invariants for subsets $U$ of $\QQ[[t]]$.  
 This section is written to be self-contained, as a paradigm for more
 general max-min characterizations in Bhargava's theory.
 We draw consequences about the $U$-sequences in $\ddR[[t]]$, and the
 integrality of $t$-binomial coefficients. 
 
   In Section \ref{sec:5} we prove the main result of the paper, the well-definedness of
 $b$-exponent sequence invariants 
 for  subsets $\SSS$ of $\ZZ$.
   This is done using ``congruence-preserving maps" $\varphi_{b}(\cdot)$
 into a larger ring $\QQ[[t]][x]$ to which Bhargava's theory applies. 
  These maps do not preserve the ring operations, addition or multiplication.
  They are given explicitly,  using  base $b$ digit expansions,
  and are  ``universal" in $\SSS$, working for all $\SSS \subseteq \ZZ$.
  In Section \ref{subsec:proofs-53} we combine results
 to  complete the proof of the well-definedness  result, Theorem \ref{thm:well-definedness}.

  In Section \ref{sec:6} we prove the results of Section \ref{sec:2} about generalized factorial ideals
 and generalized binomial coefficients for $\ZZ$.

 In Section \ref{sec:7} we treat the special case $\SSS=\ZZ$, and  calculate 
 its generalized factorials, generalized binomial coefficients and generalized integers. 
  Appendix \ref{appendix:tables} gives tables of generalized factorials 
 and generalized binomial coefficients for  $\SSS= \ZZ$.

In Section \ref{sec:7a} we treat the special case $\SSS=\PP$.

 In Section \ref{sec:9} we  discuss
 Bhargava's  generalization of $p$-sequence  to Dedekind domains,
  connections to products of binomial coefficients which motivated this work,
  and further generalizations.

%
%

\section{$\tp$-orderings and $\tp$-exponent sequences  in the power series ring $\ddR[[t]]$}\label{sec:Bhargava-Z[[t]]}

In this section we consider the
formal power series ring $R = \ddR[[t]]$, which has a unique nonzero prime ideal
$(t) = t \ddR[[t]]$.  We study $(t)$-orderings  for any 
nonempty subset $U$ of this ring, and their associated $t$-exponent sequences.
We establish a max-min characterization 
of the $t$-exponent sequence 
$\alpha_k(U, {\bf f})$ associated to  a $t$-ordering ${\bf f}= (f_i(t))_{i = 0}^\infty$ of $U$,
given as  Theorem \ref{thm:max-min},
which demonstrates its independence of the $t$-ordering. 
It yields  a self-contained proof of Bhargava's Theorem \ref{thm:invariant}  of the well-definedness
of $\pp$-sequences for any subset $U$ of this ring, with $\pp= (t)$. The results and proofs of
this section hold more generally for $R= \KK[[t]]$, with $\KK$ a field of any characteristic.

%
%
\subsection{$\tp$-orderings of an arbitrary subset of $\ddR[[t]]$}\label{section:t-orderings}

Let $\ddR[[t]]$ be the ring of formal power series 
$f(t)=\sum_{i=0}^\infty\dda_it^i$
over $\ddR$, 
with additive valuation $\dnu:\ddR[[t]]\rightarrow\NN\cup\{+\infty\}$ given by $\dnu(0) = +\infty$ and
$$
\dnu(f(t))=\sup\left\{\alpha\in\mathbb{N}:t^\alpha\text{ divides }f(t)\right\}.
$$

%
%
\begin{definition}\label{t-ordering-definition}
Let $\ddS$ be a nonempty subset of $\ddR[[t]]$.
A \emph{$\tp$-ordering} of $\ddS$ is a sequence $\mathbf{f}\coloneqq\left(f_i(t)\right)_{i=0}^\infty$ of formal power series in $\ddS$ that is formed as follows:
\begin{enumerate}
\item
  Choose any formal power series $f_0(t)$ in $\ddS$.
\item Suppose that $f_j(t)$, $j=0,1,2,\dots,k-1$ are chosen. Choose $f_k(t)$ in $\ddS$ that minimizes
$$
\sum_{j=0}^{k-1}\dnu\left(f_k(t)-f_j(t)\right).
$$

\end{enumerate}

The {\em $t$-exponent sequence} $(\ddnu_k(\ddS,\mathbf{f}))_{k = 0}^{\infty}$ associated with a $t$-ordering $\mathbf{f}= (f_i(t))_{i= 0}^{\infty}$ is
$$
\ddnu_k(\ddS,\mathbf{f})= \sum_{j=0}^{k-1}\dnu\left(f_k(t)-f_j(t)\right).
$$

\end{definition}

%
%
\begin{lem}\label{lem:t-sequence-increasing}
For any nonempty subset  $\ddS$ of $\ddR[[t]]$ and a $\tp$-ordering $\mathbf{f}=\left(f_i(t)\right)_{i=0}^\infty$ of $\ddS$, the associated 
$\tp$-exponent sequence $\left(\ddnu_k(\ddS,\mathbf{f})\right)_{k=0}^\infty$ is nondecreasing.
\end{lem}

\begin{proof}
Let $\ddS\subseteq\ddR[[t]]$, and let $\mathbf{f}=\left(f_i(t)\right)_{i=0}^\infty$ be a $\tp$-ordering of $\ddS$. Let $k$ be a nonnegative integer. 

From Definition \ref{t-ordering-definition}, we have
$$
\sum_{j=0}^{k-1}\dnu\left(f_k(t)-f_j(t)\right)\le\sum_{j=0}^{k-1}\dnu\left(f_{k+1}(t)-f_j(t)\right).
$$
Hence $\ddnu_k(\ddS,\mathbf{f})\le\ddnu_{k+1}(\ddS,\mathbf{f})-\dnu\left(f_{k+1}(t)-f_k(t)\right)\le\ddnu_{k+1}(\ddS,\mathbf{f})$.
\end{proof}

%
%
\subsection{Well-definedness of  $t$-exponent sequences 
of arbitrary nonempty $U\subseteq \ddR[[t]]$; max-min-characterization}\label{section:max-min-characterization}

Our proofs rely on a special case of
 Bhargava's well-definedness  result on $p$-orderings for Dedekind rings (\cite[Theorem 1]{Bhar:97a}),
rephrased in terms of $t$-exponent sequences.

%
%
\begin{thm}
\label{thm:invariant}
For all nonempty subsets  $\ddS$  of $\ddR[[t]]$ 
and any $\tp$-ordering $\mathbf{f}=\left(f_i(t)\right)_{i=0}^\infty$ of $\ddS$
 the $\tp$-exponent sequence 
 \begin{equation}\label{t-sequence-definition}
\ddnu_k(\ddS,\mathbf{f}):=\sum_{j=0}^{k-1}\dnu\left(f_k(t)-f_j(t)\right) \in \NN \cup \{ +\infty\}
\end{equation}
is independent of the choice of $\tp$-ordering $\mathbf{f}$ of $\ddS$.
\end{thm}

Below we  establish a new characterization of the $t$-exponent sequence   $\ddnu_k(\ddS,\mathbf{f})$ of a $t$-ordering $\mathbf{f}$ 
as a  max-min expression, which does not depend on the $t$-ordering $\mathbf{f}$, thus giving
a direct proof of Theorem \ref{thm:invariant} .
The max-min characterization concerns polynomials $p(x;t)$ in $x$ with coefficients in the ring  $\ddR[[t]]$.
We make the following definition.

\begin{definition}
A polynomial $p(x;t)\in\ddR[[t]][x]$, with coefficients in $\ddR[[t]]$, 
is \emph{$t$-primitive} if $\dnu([x^i]p(x;t))=0$ for some  $i$.
Here $[x^{i}]p(x;t)$ denotes the coefficient of $x^i$ in $p(x;t)$; i.e., if $p(x;t)=\sum_{j \ge 0} c_j(t) x^j$, then $[x^i]p(x;t)=c_i(t)$.
\end{definition}

An important feature of the set of  $t$-primitive polynomials is that it is closed under multiplication; this follows from Gauss's lemma for polynomials. 

We show $\ddnu_k(\ddS,\mathbf{f})$ is given by a max-min condition on 
primitive degree $k$ polynomials  (in $x$) in the polynomial ring $\ddR[[t]][x]$.

\begin{thm}[Max-min characterization of $t$-exponent sequences]\label{thm:max-min}
Let $\mathbf{f}=\left(f_i(t)\right)_{i=0}^\infty$ be a $\tp$-ordering of $\ddS$. Then
\begin{equation}\label{eqn:nu=maxmin}
\ddnu_k(\ddS,\mathbf{f})=\max_{p(x;t)\in \Prim_k} \left[ \min\left\{\dnu(p(f(t);t)):f(t)\in\ddS\right\}\right],
\end{equation}
where the maximum runs over the set $\Prim_k$ of all $t$-primitive polynomials $p(x;t)$ of degree $k$ in $x$. In addition,
 if
$$
q_k(x;t):=\prod_{j=0}^{k-1}\left(x-f_j(t)\right),
$$
then
\begin{equation}\label{eqn:equality-qk}
\ddnu_k(\ddS,\mathbf{f})=\min\left\{\dnu\left(q_k(f(t);t)\right):f(t)\in\ddS\right\}.
\end{equation}
\end{thm}

We defer the   proof of  Theorem \ref{thm:max-min} to
Section \ref{section:proof-max-min-characterization}.

\begin{proof}[Proof~of~Theorem~\emph{\ref{thm:invariant}} assuming Theorem ~\emph{\ref{thm:max-min}}]
In Theorem \ref{thm:max-min},  the right side of \eqref{eqn:nu=maxmin} 
does not depend on the choice of $\tp$-ordering $\mathbf{f}$, so the result follows.
\end{proof}

 The max-min Theorem \ref{thm:max-min} 
 gives useful  information about $t$-exponent sequences.
It gives monotonicity of $t$-exponent sequences under increase the set $U$,  in  Theorem \ref{thm:subset-decreasing}.
We also derive a  weak majorization result for general $U$-test sequences given in Theorem \ref{thm:majorization}.
These results are inherited by $b$-exponent sequences for all $b \ge 2$, as shown in Section \ref{section:B-test-majorization}.

%
%
\subsection{Proof of max-min-characterization for $t$-exponent sequences of nonempty $\ddS\subseteq\ddR[[t]]$}\label{section:proof-max-min-characterization}

The proof  of Theorem \ref{thm:max-min}  is based on two preliminary lemmas 
concerning  polynomials $p(x;t)$ in $x$ with coefficients in $\ddR[[t]]$.

%
%
%

\begin{lem}\label{lem:polynomial}
Suppose that $\mathbf{f} =\left(f_i(t)\right)_{i=0}^\infty$ is a $\tp$-ordering of $\ddS$, and set
$q_j(x;t) \in \ddR[[t]][x]$ be the monic polynomials given by $q_0(x;t)=1$ and,
for $j \ge 1$, 
\begin{equation}\label{eqn:q-jay}
q_j(x; t) := \prod_{i=0}^{j-1}\left(x-f_i(t)\right).
\end{equation}
Let $p(x;t) \in \QQ[[t]][x]$ be any polynomial (in $x$) of degree at most $k$. Then
\begin{equation}\label{eqn:p-basis}
p(x;t):= \sum_{j=0}^k c_j(t)q_j(x;t)
\end{equation}
for unique formal power series $c_j(t)\in \QQ[[t]]$, $0 \le j \le k.$

\emph{(1)} For $0\le j \le k$, set  
$\mu_j:=\min_{f(t) \in \ddS} \dnu\left(c_j(t)q_j(f(t);t)\right).$
Then,   
\begin{equation}\label{eqn:min-nu-pj-S}
\mu_j=\dnu\left(c_j(t)\right)+\sum_{i=0}^{j-1}\dnu\left(f_j(t)-f_i(t)\right).
\end{equation}

\emph{(2)} Set 
$\mu := \min_{f(t) \in \ddS} \dnu(p(f(t);t))$.
Then, for $0 \le j \le k$,
\begin{equation}\label{eqn:min-nu-p-S}
\mu \le \mu_j.
\end{equation}
\end{lem}

\begin{proof} 
Since each $(q_j(t))_{j=0}^k$ is a  monic polynomial of degree $j$, for $0 \le j \le k $ they form a basis of 
the $\QQ[[t]]$-vector space of all polynomials of degree at most $k$ in $\QQ[[t]][x]$, so
the $c_j(t)$ are uniquely defined.

(1) 
We have
\begin{equation}\label{eqn:valuation-p}
\dnu\left(p_j(f(t);t)\right)=\dnu\left(c_j(t)\right)+\sum_{i=0}^{j-1}\dnu\left(f(t)-f_i(t)\right).
\end{equation}
By  definition of $t$-exponent sequence, the  right-hand side is minimized  over $f(t)\in \ddS$ at $f(t)=f_j(t)$.

 (2)   The  quantity 
$\mu$
 depends on $p(x; t)$,  but it does not depend on the $t$-ordering  $\mathbf{f}$.
We consider $k$ as fixed, and prove \eqref{eqn:min-nu-p-S} by induction on $j \ge 0$. 
 For the base case $j =0$, choosing $f(t)= f_0(t) \in \ddS$ we have $p_j(f_0(t); t)=0$ for $j \ge 1$ so 
$$
\mu\le\dnu\left(p\left(f_0(t);t\right)\right)=\dnu\left(p_0\left( f_0(t), t\right)\right) = \dnu\left(c_0(t)\right)= \mu_0.
$$
Hence \eqref{eqn:min-nu-p-S} is true for fixed $k$ and $j=0$. 

Now, suppose that \eqref{eqn:min-nu-p-S} holds for $j=0,\dots,\ell-1$, where $1\le\ell\le k$,
we  show  it holds for $j = \ell$.  Define $p_j(x;t) = c_j(t) q_j(x;t).$
Then using \eqref{eqn:valuation-p} we  can rewrite property (1)
for the choice $j=\ell$  as
 asserting
\begin{equation}\label{eqn:min-nu-pj-S-new}
\mu_{\ell} = \min_{f(t) \in \ddS} \dnu\left(p_\ell(f(t);t)\right)=\dnu\left(p_\ell\left(f_\ell(t);t\right)\right).
\end{equation}
From the identity
$$
p_\ell\left(f_\ell(t);t\right)=p\left(f_\ell(t);t\right)-\sum_{j=0}^{\ell-1}p_j\left(f_\ell(t);t\right)
$$
and the valuation property of $\dnu$, we deduce 
\begin{align*}
\dnu\left(p_\ell\left(f_\ell(t);t\right)\right)&\ge\min\left\{\dnu\left(p\left(f_\ell(t);t\right)\right)\right\}\cup\left\{\dnu\left(p_j\left(f_\ell(t);t\right)\right):0\le j\le\ell-1\right\}\\
&\ge \min\,\{ \mu \}\cup\left\{\min_{f(t)\in\ddS}[ \dnu\left(p_j(f(t);t)\right)]:\,0\le j\le\ell-1\right\}  \\
&= \min\, \{\mu, \mu_0, \mu_1, \dots, \mu_{\ell-1}\} = \mu,
\end{align*}
with  the induction hypothesis used on the last line.
 From \eqref{eqn:min-nu-pj-S-new}, we conclude that $\mu \le \mu_{\ell}$, completing
 the induction step.
\end{proof}

We first show each $t$-primitive polynomial gives a lower bound for $\ddnu_k(\ddS, \mathbf{f})$.
%
%
\begin{prop}\label{prop:Polya}
Let $p(x;t)$ be a $t$-primitive polynomial in $x$ of degree $k$, and $\ddS$ a subset of $\ddR[[t]]$.
Then for any $t$-ordering $\mathbf{f}$ for $\ddS$, we have
$$
\min\left\{\dnu(p(f(t);t)):f(t)\in\ddS\right\}\le\ddnu_k(\ddS,\mathbf{f}).
$$
\end{prop}

\begin{proof}
Let $\mathbf{f}=\left(f_i(t)\right)_{i=0}^\infty$ be a $\tp$-ordering of $\ddS$. 
We let $q_j(x;t) = \prod_{i=0}^{j-1}(x- f_j(t))$ for $0 \le j \le k$.
Let  $p(x;t)= \sum_{j=0}^k \tilde{c}_j(t) x^j$ be an arbitrary
polynomial of degree $k$. By Lemma \ref{lem:polynomial}
it has  a unique decomposition of  the form 
$p(x;t) = \sum_{j=0}^{k} c_j(t) q_j(x; t).$

 Now suppose  $p(x;t)$ is $t$-primitive, so $p(x;t) = \sum_{j=0}^k \tilde{c}_j(t) x^j$ has
 some coefficient $\tilde{c}_{j}(t)$
 with $\dnu( \tilde{c}_j(t))=0$.
 It now follows there is some $j_0 \in \{0,\dots,k\}$ such that 
 $\dnu\left(c_{j_0}(t)\right)=0$. 
 Specifically it holds with $j_0$ being the largest $j$ for which $\dnu  (\tilde{c}_j(t)) =0.$
 (The change of basis from $\{x^j: 0 \le j \le k\}$ to $\{ q_j(x; t): 0 \le j \le k\}$ is upper triangular unipotent,
 proceeding from degree $k$ downwards.)

 Applying Lemma \ref{lem:polynomial} with $j=j_0$, we obtain
\begin{align*}
\min_{f(t) \in \ddS}\dnu(p(f(t);t))=\mu &\le \mu_{j_0}= \min_{f(t) \in \ddS} \dnu\left(p_{j_0}(f(t);t)\right)\\
&=\dnu\left(c_{j_0}(t)\right)+\sum_{i=0}^{j_0-1}\dnu\left(f_{j_0}(t)-f_i(t)\right)\\
&=\sum_{i=0}^{j_0-1}\dnu\left(f_{j_0}(t)-f_i(t)\right)\\
&=\ddnu_{j_0}(\ddS,\mathbf{f}).
\end{align*}
The last quantity is $\le\ddnu_k(\ddS,\mathbf{f})$ by Lemma \ref{lem:t-sequence-increasing}.
\end{proof}

\begin{proof}[Proof of Theorem \ref{thm:max-min}]
Taking the maximum over all $t$-primitive polynomials $p(x;t)$ of degree $k$ in $x$ in Proposition \ref{prop:Polya}, we obtain
$$
\ddnu_k(\ddS,\mathbf{f})\ge\max_{p(x;t) \in \Prim_k}\left[ \min_{f(t) \in \ddS} \dnu(p(f(t);t)) \right].
$$

To  establish \eqref{eqn:nu=maxmin}, which says
$$
\ddnu_k(\ddS,\mathbf{f})=\max_{p(x;t)\in \Prim_k} \left[ \min_{f(t) \in \ddS} \dnu(p(f(t);t)) \right],
$$
 it  suffices to prove the second assertion \eqref{eqn:equality-qk},
 which says
$$
\ddnu_k(\ddS,\mathbf{f})=\min_{f(t) \in \ddS} \dnu\left(q_k(f(t);t)\right)
$$ 
 holds  for all $k \ge 0$.
  The truth of  \eqref{eqn:equality-qk} for a given $\mathbf{f}$ implies, since $q_k(x;t)$ is a $t$-primitive polynomial in $x$ of degree $k$, 
 \begin{eqnarray}
 \ddnu_k(\ddS,\mathbf{f}) &=& \min_{f(t) \in \ddS} \dnu\left(q_k(f(t);t)\right)\\
 &\le & \max_{p(x;t)\in \Prim_k} \min_{f(t) \in \ddS} \dnu(p(f(t);t)).
 \end{eqnarray}
 We conclude that \eqref{eqn:nu=maxmin} holds.
 
 To verify \eqref{eqn:equality-qk},
by definition of 
 $\tp$-ordering $\mathbf{f}$, we have
\begin{align*}
\min\left\{\dnu\left(q_k(f(t);t)\right):f(t)\in\ddS\right\}&=\min_{f(t)\in\ddS}\sum_{j=0}^{k-1}\dnu\left(f(t)-f_j(t)\right)\\
&=\sum_{j=0}^{k-1}\dnu\left(f_k(t)-f_j(t)\right)\\
&=\ddnu_k(\ddS,\mathbf{f}),
\end{align*}
which establishes  \eqref{eqn:equality-qk}. This completes the proof.
\end{proof}

%
%

\subsection{Properties  of the associated $\tp$-exponent sequences}\label{subsection:properties-t-sequence}

We have  shown that the value $\ddnu_k(\ddS,\mathbf{f})$ does not depend on 
the $t$-ordering $\mathbf{f}$, and we will henceforth denote it as $\ddnu_k(\ddS)$.
We  call $(\ddnu_k(\ddS))_{k=0}^{\infty}$ the  \emph{$\tp$-exponent sequence of $U\subseteq\ddR[[t]]$}.

%
%
\begin{thm}\label{thm:binomial}
For $U \subseteq \ddR[[t]]$, and nonnegative integers $k$ and $\ell$, we have
\begin{equation}
\ddnu_{k+\ell}(\ddS )\ge\ddnu_k(\ddS)+\ddnu_\ell(\ddS ).
\end{equation}
\end{thm}

\begin{proof}
Let $\mathbf{f}=\left(f_i(t)\right)_{i=0}^\infty$ be a $\tp$-ordering of $\ddS$. 
Applying Theorem \ref{thm:max-min}, we obtain
\begin{align*}
\ddnu_k(\ddS)+\ddnu_\ell(\ddS)&=\min_{f(t) \in \ddS}\dnu\left(q_k(f(t);t)\right)+\min_{f(t) \in \ddS}\dnu\left(q_\ell(f(t);t)\right)\\
&\le\min_{f(t) \in \ddS} \left(\dnu\left(q_k(f(t);t)\right)+\dnu\left(q_\ell(f(t);t)\right)\right).
\end{align*}
Then
\begin{align*}
\ddnu_k(\ddS)+\ddnu_\ell(\ddS)&\le\min_{f(t) \in \ddS} \left(\dnu\left(q_k(f(t);t)q_\ell(f(t);t)\right)\right)\\
&\le\max_{p(x;t)\in \Prim_{k+\ell}}\min_{f(t) \in \ddS} \dnu(p(f(t);t))=\ddnu_{k+\ell}(\ddS),
\end{align*}
in which the last inequality used the property that products of $t$-primitive polynomials are $t$-primitive polynomials.
\end{proof}

%
%
\begin{thm}\label{thm:subset-decreasing}
Suppose that $\ddS_1\subseteq\ddS_2\subseteq \ddR[[t]]$. Then 
\begin{equation}
\ddnu_k(\ddS_1)\ge\ddnu_k(\ddS_2)
\end{equation}
holds  for each 
$k \ge 0$.
\end{thm}

\begin{proof}
Let $p(x;t)$ be a $t$-primitive polynomial of degree $k$. Since $\ddS_1\subseteq\ddS_2$, it follows that
$$
\min_{f(t) \in \ddS_1} \dnu(p(f(t);t))\ge\min_{f(T) \in \ddS_2} \dnu(p(f(t);t)).
$$
Taking the maximum over all $t$-primitive polynomials $p(x;t)$ of degree $k$ and applying Theorem \ref{thm:max-min}, we obtain 
$\ddnu_k(\ddS_1 )\ge\ddnu_k(\ddS_2  )$.
\end{proof}

%
%
\begin{thm}\label{thm:majorization}
Let $\mathbf{g} := ( g_i(t))_{i=0}^{\infty}$
be any $\ddS$-test sequence in $\QQ[[t]]$.
Then for each $n \ge 1$ there holds 
\begin{equation}\label{eqn:superfactorial-integer}
\sum_{k=1}^n \ddnu_{k}(\ddS,  \mathbf{g}) \ge \sum_{k=1}^n\ddnu_k(\ddS).
\end{equation}
Given a fixed $N \ge 1$, equality holds in \eqref{eqn:superfactorial-integer} for all $0 \le n \le N$ 
whenever  $(g_i(t))_{i=0}^N$ is the initial part of a $t$-ordering for $\ddS$.
\end{thm}

\begin{proof}
For any $\ddS$-test sequence $\mathbf{g}$,  we have by definition
\begin{align}\label{eqn:test-double-sum}
\sum_{k=1}^n \ddnu_{k}(\ddS,  \mathbf{g}) &= \sum_{k=1}^n \sum_{j=0}^{k-1} \dnu(g_k(t) - g_j(t)).
\end{align}
We then obtain
\begin{align}\label{eqn:permutation-inv}
\sum_{k=1}^n \ddnu_{k}(\ddS,  \mathbf{g}) &= \frac{1}{2} \bigg( \sum_{ \{(j, k): j \ne k, 0 \le j, k \le n\}} \dnu(g_k(t) - g_j(t))\bigg)
\end{align}
since $\dnu(g_k(t)- g_j(t)) = \dnu(g_j(t)- g_k(t))$. Since the right side of \eqref{eqn:permutation-inv} is 
invariant under permutations $ \sigma \in S_n$ sending  $g_k(t)$ to $g_{\sigma(k)}(t)$, 
the  sum $\sum_{k=1}^n \ddnu_{k}(\ddS,   \mathbf{g})$ is independent of the order
of the elements in the finite set $V= \{ g_k(t)\}_{0 \le k \le n}$.

Now  let $\mathbf{g}' := (g_0^{'}, g_1^{'}, \dots , g_n^{'})$ be a permutation of these elements that is a
$t$-ordering for the finite set $V$. We deduce
$\sum_{k=1}^n \ddnu_{k}(\ddS,  \mathbf{g}) = \sum_{k=1}^n \ddnu_{k}(V,   \mathbf{g}').$
Thus 
\begin{equation}\label{eqn:UV}
\sum_{k=1}^n \ddnu_{k}(\ddS,  \mathbf{g}) =  \sum_{k=1}^n \ddnu_{k}(V)  \ge  \sum_{k=1}^n \ddnu_{k}(\ddS),
\end{equation}
the left  equality holding by Theorem \ref{thm:invariant}, applied with $\ddS= V$, since $\mathbf{g}'$ is
a $t$-ordering,  and the  inequality on the right holding 
term by term using  Theorem \ref{thm:subset-decreasing}, with $\ddS_1=V$ and $\ddS_2= \ddS$, 
since $V \subseteq \ddS$.

Finally, we see that equality holds  in  \eqref{eqn:superfactorial-integer} whenever the   $\ddS$-test sequence
$\{ s_i: 0 \le i \le n\}$ is the
initial part of a $t$-ordering of $\ddS$, since we may take $V=U$ in \eqref{eqn:UV}.
\end{proof}

\begin{remark}
For fixed $N\ge 1$ and the inequalities  in Theorem \ref{thm:majorization}  for $1 \le k \le N$ assert:
For  any $\ddS$-test sequence $\mathbf{g}= ( g_i(t))_{i = 0}^\infty$,
the sequence of initial $t$-exponents  $\{ \ddnu_k(\ddS,  \mathbf{g})\}_{k=1}^N$
{ \em weakly majorizes} the sequence of initial 
$t$-exponents $( \ddnu_k(\ddS))_{k=1}^N$ 
  attached to $\ddS$.  However the  associated $t$-exponent sequences  for general test sequences
 need  not satisfy the  majorization condition of equal sums at step $N$.
 Majorization inequalities have 
 geometric interpretations
 and many applications, see the books of 
Marshall and Olkin \cite{MO:78}, and Marshall, Olkin and Arnold  \cite{MOA:11}.
\end{remark}

%
%
\section{Congruence-preserving mappings}\label{sec:5} 

In this section we prove the main result (Theorem \ref{thm:well-definedness})
for each 
$b \ge 2$, by constructing and using an embedding by
a ``congruence preserving" map  from $\ZZ$ into the ring $\ZZ[[t]]$. We view $\ZZ[[t]]$
as a subring of the Dedekind domain $\ddR[[t]]$, to which the results of Section \ref{sec:Bhargava-Z[[t]]} apply.
The  cases $b=0$ and $b=1$ are  handled directly.

%
%
\subsection{Congruence-preserving property for  ideals $(b)$ of $\mathbb{Z}$}\label{subsec:proofs-Z}

Let $\bb$ be a nonzero proper  ideal of $\mathbb{Z}$, so $\bb= (b)$ for some integer $b \ge 2$.
We construct a congruence-preserving  map of $\ZZ$ into $\QQ[[t]]$, using base $b$ digit expansions. By
a {\em congruence-preserving map}$\pmod{b}$ we mean a mapping $\varphi_b: \ZZ \to \QQ[[t]]$ such that
for all $a_1, a_2 \in \ZZ$, 
\begin{equation}\label{eqn:equal-valuations1}
\dnu\left(\varphi_b\left(a_1\right)-\varphi_b\left(a_2\right)\right)=\order_b\left(a_1-a_2\right).
\end{equation}
Equivalently,
for all $a_1, a_2 \in \mathbb{Z}$,
\begin{equation}\label{map-condition}
\varphi_b\left(a_1\right)\equiv\varphi_b\left(a_2\right)\pmod{t^k}\quad\text{if and only if}\quad a_1\equiv a_2\pmod{b^k}
\end{equation}
for all $k\ge1$.

%
%
\begin{prop}\label{prop:varphi_B-exists-for-Z}
Given $b \ge 2$, let   $\varphi_b:\mathbb{Z}\to\mathbb{Q}[[t]]$ be defined 
for each $a \in \mathbb{Z}$ by
\begin{equation}\label{eqn:definition-Z-digits-1}
\varphi_b(a)=f_{a,b}(t):=\sum_{k=0}^\infty d_kt^k,
\end{equation}
where $a = \sum_{k \ge0} d_k b^k$ is the $b$-adic expansion of $a$, using the  
digit set $\sD = \{ 0, 1, 2, \ldots, b-1\}$, 
which is  computed recursively by 
\begin{equation}\label{eqn:definition-Z-digits-2}
d_k=d_k(a,b):=\left\lfloor\frac{a}{b^k}\right\rfloor-b\left\lfloor\frac{a}{b^{k+1}}\right\rfloor.
\end{equation}
Then $\varphi_b(\cdot)$ is a congruence-preserving map$\pmod{b}$ of the ring $\ZZ$ to
the formal power series ring $\QQ[[t]]$, whose image is contained in the subring $\ZZ[[t]]$.
\end{prop}

Note that for $a \ge 0$ the  $b$-adic expansion \eqref{eqn:definition-Z-digits-1} of $a$ has finitely many nonzero digits $d_k$, 
while for negative integers $a \le -1$, it has  infinitely many nonzero digits in its $b$-adic expansion,
so  the image of the map $\varphi_b (\cdot)$  includes infinite power series.

\begin{proof}

($\Rightarrow$) Suppose that $\varphi_b\left(a_1\right)\equiv\varphi_b\left(a_2\right)\pmod{t^k}$. That is,
$d_\ell \left(a_1,b\right)=d_\ell\left(a_2,b\right)$ for all $\ell=0,\dots,k-1$. From the identity
$$
a-b^k\left\lfloor\frac{a}{b^k}\right\rfloor=\sum_{\ell=0}^{k-1}d_\ell(a,b)b^\ell,
$$
we see that $a_1-b^k\left\lfloor\frac{a_1}{b^k}\right\rfloor=a_2-b^k\left\lfloor\frac{a_2}{b^k}\right\rfloor$. We deduce that  
$a_1\equiv a_2\pmod{b^k}$.

($\Leftarrow$) Suppose that $a_1\equiv a_2\pmod{b^k}$. That is, $a_1-a_2=b^kq$ for some $q\in\mathbb{Z}$. So for $\ell=0,\dots,k-1$,
\begin{align*}
d_\ell \left(a_1,b\right)&=\left\lfloor\frac{a_1}{b^\ell}\right\rfloor-b\left\lfloor\frac{a_1}{b^{\ell+1}}\right\rfloor\\
&=\left(\left\lfloor\frac{a_2}{b^\ell}\right\rfloor+b^{k-\ell}q\right)-b\left(\left\lfloor\frac{a_2}{b^{\ell+1}}\right\rfloor+b^{k-\ell-1}q\right)\\
&=\left\lfloor\frac{a_2}{b^\ell}\right\rfloor-b\left\lfloor\frac{a_2}{b^{\ell+1}}\right\rfloor=d_\ell\left(a_2,b\right).
\end{align*}
Hence $\varphi_b\left(a_1\right)\equiv\varphi_b\left(a_2\right)\pmod{t^k}$.
\end{proof}

%
%
\begin{rem}\label{rem:41}

(1) The  congruence-preserving maps $\varphi_{b}$   from  
the ring $\ZZ$ to  $\QQ[[t]]$ are not homomorphisms for  either of 
the  ring operations: ring  addition and  ring multiplication. 
They do preserve both $0$ and $1$.

(2) 
In general a congruence-preserving map (property \eqref{eqn:equal-valuations1}) need not preserve $0$ or $1$.
The equivalent condition \eqref{map-condition}
 for all $k \ge 1$ implies  that  any congruence-preserving map $\varphi_{b}$ must be  injective, because
$ \bigcap_{k=0}^{\infty} t^k \QQ[[t]] = \{0\}$
holds.
\end{rem}

%
%
\subsection{Well defined $b$-exponent sequence invariants for  $\ZZ$}\label{subsec:52}

%
%
\begin{thm}
\label{thm:main-propertyCZ}
For fixed $b \ge 2$,  for any nonempty set $\SSS$ of $\ZZ$, all $b$-orderings $\mathbf{a}$ of $\SSS$ give the same values
$\ivr_{k}(\SSS, b, \mathbf{a})$, 
so the associated $b$-exponent sequence $\alpha_k(\SSS, b)$ is well-defined
for each $k \ge 1$.
 In addition
 \begin{equation}\label{eqn:equal-inv}
\ivr_{k} \left(\SSS, b \right)= \ivr_k ( \varphi_{b}(\SSS)),
\end{equation} 
holds for all $k \ge 1.$
\end{thm}

\begin{proof}
We are given nonempty $S$ and  $b \ge 2$ and set $U=\varphi_b(S)=\{ \varphi_b(a'): \, a' \in \SSS\} \subset \ddR[[t]]$.
Let $\mathbf{a}= (a_i)_{ 0 \le i \le |\SSS|}$ be a $b$-ordering of $\SSS$, and set 
$u_i = \varphi_{b}(a_i)$.  We claim that $\mathbf{u}= (u_i)_ {0 \le i < |U|} =\varphi_{b}(\mathbf{a})$
is a $t$-ordering of $U$.
To see this, for any  $u' \in U$, there is  some $a'\in S$ with  $u' = \varphi_{b}(a')$, 
and we have 
\begin{eqnarray*}
\sum_{j=0}^{k-1} \dnu \left( u' - u_j\right)   
&=&  \sum_{j=0}^{k-1} \ordr_b \left( a' - a_j  \right)\\
&\ge& \sum_{j=0}^{k-1} \ordr_b \left( a_k - a_j  \right)
= \sum_{j=0}^{k-1} \dnu \left( u_k - u_j  \right),\\
\end{eqnarray*} 
with the first and last equalities  holding by the congruence preserving property, 
and the  inequality  by
the  $b$-ordering property. Thus $\mathbf{u}$ is a $t$-ordering of $U$.

By Theorem \ref{thm:invariant} 
the invariant $\ivr_k (U)=  \ivr_k(U, \mathbf{u})$ is independent of the
choice of  $t$-ordering ${\bf u}$ of $U$. Thus
\begin{equation}
\label{eqn:independence-of-a} 
\ivr_k(\SSS, b, {\bf a}) 
= \sum_{j=1}^{k-1} \ordr_b \left( a_k - a_j  \right) =  \sum_{j=0}^{k-1} \dnu \left( \varphi_{b}(a_k) - \varphi_{b} (a_j) \right)=  \alpha_k (U ).
\end{equation}
Since $\mathbf{a}$ is an arbitrary $b$-ordering of $\SSS$,
 we conclude that $\ivr_k(\SSS, b, {\bf a}) $ is independent of the choice of $b$-ordering, and by \eqref{eqn:independence-of-a} 
  its value is $\ivr_k(\varphi_b(\SSS)).$
\end{proof}

%
%
\subsection{Properties of $b$-exponent sequence invariants}\label{section:B-orderings-properties}

The next proposition gives three  consequences of Theorem \ref{thm:main-propertyCZ}.
By \eqref{eqn:equal-inv} $b$-exponent sequences inherit all the properties of $\tp$-exponent sequences.

%
%
\begin{prop}\label{prop:4-prop}

{\em (1)} For any subset $\SSS$ of $\ZZ$ the associated $b$-exponent sequence 
$\left(\ddnu_k\left(\SSS,b \right)\right)_{k=0}^\infty$ is weakly increasing, i.e.
$\ddnu_{k+1} \left(\SSS,b \right) \ge \ddnu_{k} \left(\SSS,b \right)$, for all $k\ge0$.

{\em (2)} For any subset $\SSS$ of $\ZZ$ and any nonnegative integers $k$ and $\ell$, we have
$$
\ddnu_{k+\ell}\left(\SSS,b \right)\ge\ddnu_k\left(\SSS,b \right)+\ddnu_\ell\left(\SSS,b \right).
$$

{\em (3)} Suppose that $\SSS_1\subseteq\SSS_2\subseteq \ZZ$.
 Then $\ddnu_k\left(\SSS_1,b \right)\ge\ddnu_k\left(\SSS_2,b\right)$ holds  for every nonnegative integer $k$.

\end{prop}

\begin{proof}

(1) Applying Lemma \ref{lem:t-sequence-increasing} with $\ddS=\varphi_b\left(\SSS\right)$, we see that 
the associated $\tp$-exponent sequence $\left(\ddnu_k\left(\varphi_b\left(\SSS\right)\right)\right)_{k=0}^\infty$ is weakly increasing.
This sequence is the associated $b$-exponent sequence $\left(\ddnu_k\left(\SSS,b\right)\right)_{k=0}^\infty$ by 
\eqref{eqn:equal-inv} in Theorem \ref{thm:main-propertyCZ}. 

(2)  Applying Theorem \ref{thm:binomial} with $\ddS=\varphi_b\left(\SSS\right)$, we obtain
$$
\ddnu_{k+\ell}\left(\varphi_b \left(\SSS\right)\right)\ge \ddnu_k\left(\varphi_b \left(\SSS\right)\right)+\ddnu_\ell\left(\varphi_b \left(\SSS\right)\right).
$$
By  \eqref{eqn:equal-inv} in Theorem \ref{thm:main-propertyCZ}, the above yields
 $\ddnu_{k+\ell}\left(\SSS,b\right)\ge\ddnu_k\left(\SSS,b\right)+\ddnu_\ell\left(\SSS,b \right)$.

(3) Applying Theorem \ref{thm:subset-decreasing} with $\ddS_1=\varphi_b\left(\SSS_1\right)$ and 
$\ddS_2=\varphi_b\left(\SSS_2\right)$, we obtain
$$
\ddnu_k\left(\varphi_b\left(\SSS_1\right)\right)\ge\ddnu_k\left(\varphi_b\left(\SSS_2\right)\right).
$$
By \eqref{eqn:equal-inv} in Theorem \ref{thm:main-propertyCZ},
 the above is $\ddnu_k\left(\SSS_1,b\right)\ge\ddnu_k\left(\SSS_2,b\right)$.
\end{proof}
%
%
\subsection{Proof of Main Theorem ~\ref{thm:well-definedness}}\label{subsec:proofs-53}

\begin{proof}[Proof~of~Theorem~\ref{thm:well-definedness}]
For $b \ge 2$ the  result follows from  Theorem \ref{thm:main-propertyCZ} obtained
using the map $\varphi_{b} (\cdot)$  given in  Proposition  \ref{prop:varphi_B-exists-for-Z}.

It remains to  treat the cases $b=0$ and $b=1$. 
We use the  modified definition
\begin{equation}\label{eqn:ord-b2}
\ordr_{b}(a) :=\sup \{k \in \NN \, : \,  a\ZZ \subseteq  b^k\ZZ\},
\end{equation}
for all  $b \ge 0$ and all $a \in \ZZ$.
This definition agrees with \eqref{eqn:ord-b} when $b \ge 2$. 
Combining it with the convention $0^0=1$ yield for $b=0$,
$\ordr_0(a) = +\infty$ when $a=0$, $\ordr_0(a)= 0$ when $a \ne 0$,
 and  for  $b=1$ yield   $\ordr_1(a)= +\infty$ for all $a \in \ZZ$.

For $b=0$ the
$b$-orderings $\mathbf{a}= (a_i)_{i=0}^{\infty}$ consist of any sequence  of distinct elements in $\SSS$,
which can be done  for $0 \le i < |\SSS|$, while elements $a_i$ for $i \ge |\SSS|$ may be chosen arbitrarily.
 All such sequences  have the same $0$-exponent sequence:
\begin{equation*}
\alpha_k( \SSS, 0) = \begin{cases} 
   0 \quad \mbox{for} \quad 0 \le k < |\SSS|\\
 +\infty \quad \mbox{for} \quad k \ge  |\SSS|,\\
\end{cases}
\end{equation*}
proving Theorem \ref{thm:well-definedness} for $b=0$.

For $b=1$ all  sequences of elements in $\SSS$ are  $1$-orderings. They all have
 $1$-exponent sequence $\alpha_0(S,1)=0$ and
\begin{equation*}
\alpha_k( \SSS, 1) = 
+\infty \quad  \mbox{for all} \quad k \ge 1,
\end{equation*}
proving Theorem \ref{thm:well-definedness} for $b=1$.
\end{proof} 

\begin{rem}\label{rem:55}
We verify the inequalities   \eqref{eq:extreme} in Section \ref{sec:12}. Let $b \ge 2$.
The  inequalities are  true for $i=0$ by definition, so let $i \ge 1$.
Since $\ivr_{i}(\SSS, 1) = +\infty$ for all $i\ge1$ we have $\ivr_i(\SSS, b) \le \ivr_i(\SSS, 1)$
for all $\SSS$ and all $i\ge 0$. We have $\ivr_i(S,0)=0$ for $0 \le i < |\SSS|$, hence
$\ivr_i(\SSS, 0) \le \ivr_i(\SSS, b)$ for $0 \le i < \SSS$. Now $\ivr_i(\SSS, b) = +\infty$
for all $i \ge |\SSS|$ because all $\SSS$-test sequences of length $i+1$ contain a repeated
element. Hence $\ivr_i(\SSS, 0) \le \ivr_i(\SSS, b)$  holds for $i \ge |\SSS|$.
\end{rem}

%
%
\subsection{Proof of Theorem \ref{thm:majorization-b}}\label{section:B-test-majorization}

\begin{proof}
By definition
\begin{equation}\label{eqn:superfactorial-integer2}
\sum_{k=1}^{n} \alpha_{k}(\SSS, b, \mathbf{s})      :=\sum_{k=1}^{n}\sum_{j=0}^{k-1} \ordr_b\left(s_k-s_j\right)
\end{equation}
We note that, by convention, $\ivr_0( \SSS, b, \mathbf{s}) = \ivr_0(\SSS, b) =0$.

(1) Applying Theorem \ref{thm:majorization} with 
$\ddS=\varphi_b \left(\SSS\right)$ and $g_i(t)=\varphi_b \left(s_i\right)$, we obtain,
using \eqref{eqn:test-double-sum}, 
$$
\sum_{k=1}^{n}\sum_{j=0}^{k-1} \dnu\left(\varphi_b \left(s_k\right)-\varphi_b\left(s_j\right)\right)
\ge \sum_{k=1}^n\ddnu_k\left(\varphi_b \left(\SSS\right)\right).
$$
Applying \eqref{eqn:equal-valuations1}, which says 
$\dnu\left(\varphi_b\left(s_k \right)-\varphi_b\left(s_j \right)\right)=\ordr_b\left( s_k-s_j \right),$
to the left side,  
and  the equality \eqref{eqn:equal-inv} in Theorem \ref{thm:main-propertyCZ}, which says
$\ivr_{k} \left(\SSS, b \right)= \ivr_k ( \varphi_{b}(\SSS))$,
to the right side, 
yields 
\begin{equation*}\label{eqn:majorization-2}
\sum_{k=1}^{m} \alpha_{k}(\SSS, b, \mathbf{s})   
  :=\sum_{k=1}^{n}\sum_{j=0}^{k-1} \ordr_b\left(s_k-s_j\right)
  \ge\sum_{k=1}^m\ddnu_k\left(\SSS,b \right).
\end{equation*}
The case $n= +\infty$ follows as a limit of finite $n$.

(2) The sufficient condition for equality  for $1 \le n \le N$   follows also from Theorem \ref{thm:majorization} as well.
\end{proof}

%
%
\section{Generalized factorials 
 and  generalized  binomial coefficients associated to subsets $\sT$ of $\NN$}\label{sec:6}

This section  treats  generalized integers, 
generalized factorials, and
generalized binomial coefficients 
 for $\ZZ$. 
Generalized integers and generalized binomial coefficients  are, by their definition, rational numbers.
We  prove  results stated in Section \ref{sec:2}, deducing that generalized-integers 
and generalized binomial
coefficient ideals are integers .

%
%
\subsection{Product formulas  for generalized factorials}\label{subsec:62} 

We give product formulas for generalized
integers  and generalized binomial coefficients for arbitrary $(\ZZ, \sT)$,
expressing them as positive rational numbers. 
These product formulas have some factors $b$ that are  composite numbers. Obtaining  prime factorizations 
of these generalized factorials requires making further prime factorization of the composite terms $b$ that occur.
%
%
\begin{lem}\label{lem:61} 
Let $S$ be a nonempty subset of $\ZZ$ 
and let  $\T\subseteq \NN$.
Then for positive integers $1 \le n<\vert S\vert$,
the generalized-integers 
$[n]_{S,\T}$ are given by
\begin{equation}\label{eqn:integer-factorization-formula}
[n]_{S,\T}=\prod_{b\in\T}b^{\ivr_n(S,b)-\ivr_{n-1}(S,b)}.
\end{equation}
In particular, each such $[n]_{S,\T}$ is a positive integer.
\end{lem}

\begin{proof}
The right side of the  formula \eqref{eqn:integer-factorization-formula}
(viewed as a  positive rational number) 
follows from Definition \ref{definition:integers-S-T} and Definition  \ref{definition:factorials-S-T}.
However we  know each term $b^{\ivr_n(S,b)-\ivr_{n-1}(S,b)}$
  is an integer  since 
$\ivr_n(S,b)-\ivr_{n-1}(S,b)\ge 0$
holds for all $n \ge 1$,
by Proposition \ref{prop:varphi_B-exists-for-Z} combined with 
Proposition \ref{prop:4-prop} (1).
\end{proof}

%
%
\begin{lem}\label{lem:62}
Let $S$ be a nonempty subset of  $\ZZ$. Let $\T\subseteq \NN$.
Then for integers $0\le\ell\le k<\vert S\vert$,
the generalized binomial coefficients  $\binom{k}{\ell}_{S,\T}$  are given by
\begin{equation}\label{eqn:binomial-factorization-formula}
\binom{k}{\ell}_{S,\T}=\prod_{b\in\T} b^{\ivr_k(S,b)-\ivr_\ell(S,b)-\ivr_{k-\ell}(S,b)}.
\end{equation}
In particular, each $\binom{k}{\ell}_{S,\T}$ is a positive integer.
\end{lem}

\begin{proof}
The formula \eqref{eqn:binomial-factorization-formula}, with 
the right side viewed as a positive rational number, 
follows from Definition \ref{definition:binomial-S-T} and Definition \ref{definition:factorials-S-T}.
We  know each term $b^{\ivr_k(S,b)-\ivr_\ell(S,b)-\ivr_{k-\ell}(S,b)}$
 is an integer  since 
$\ivr_k(S,b)-\ivr_\ell(S,b)-\ivr_{k-\ell}(S,b) \ge 0$,
holds for all $k \ge \ell \ge 0$ by Proposition \ref{prop:varphi_B-exists-for-Z} combined with 
Proposition \ref{prop:4-prop} (2).
\end{proof}

%
%
\subsection{Proofs of inclusion  and integrality properties}\label{subsec:61} 

We prove Proposition \ref{prop:factorial-ordering}, Theorem \ref{thm:integer-integer} and
Theorem \ref{thm:binomial-integer}.
%
%
\begin{proof}[Proof~of~Proposition~\emph{\ref{prop:factorial-ordering}}]
(1) The divisibility property  holds because
$$
k!_{S,\T_2}=k!_{S,\T_2\backslash\T_1}k!_{S,\T_1},
$$
and the generalized factorials $k!_{S, \T}$ are 
integers.

(2) 
Proposition \ref{prop:varphi_B-exists-for-Z}  
shows that the hypotheses of Proposition \ref{prop:4-prop} (3)
hold, hence
there holds 
$$\ivr_k\left(S_1, b\right)\ge\ivr_k\left(S_2, b\right)$$
 for each $b \in\NN$.
  Since
$$
\frac{k!_{S_1,\T}}{k!_{S_2,\T}}=\prod_{b\in\T}b^{\ivr_k\left(S_1,b\right)-\ivr_k\left(S_2,b\right)},
$$
the right side  is an integer.
\end{proof}
%
%

\begin{proof}[Proof~of~Theorem~\emph{\ref{thm:integer-integer}}]
It follows from Proposition \ref{prop:varphi_B-exists-for-Z} and Proposition \ref{prop:4-prop} (1)
 that
$$
\ivr_n(S,b)\ge\ivr_{n-1}(S,b)\quad\mbox{for all}\quad b\in \NN.
$$
The result that $[n]_{\SSS, \sT}$ is a positive  integer follows from Lemma \ref{lem:61} using
\eqref{eqn:integer-factorization-formula}.
\end{proof}

%
%

\begin{proof}[Proof~of~Theorem~\emph{\ref{thm:binomial-integer}}]
It follows from Proposition \ref{prop:varphi_B-exists-for-Z} and Proposition \ref{prop:4-prop} (2) 
that
$$
\ivr_k(\SSS,b)\ge\ivr_\ell(\SSS,b)+\ivr_{k-\ell}(\SSS,b)\quad\mbox{for all}\quad b\in \NN.
$$
The result that $\binom{k}{\ell}_{\SSS,\T}$ is a positive  integer  then follows from Lemma \ref{lem:62} using 
\eqref{eqn:binomial-factorization-formula}.
\end{proof}

%
%
\section{The  case $\SSS= \ZZ$ }\label{sec:7} 

We treat generalized factorials, generalized binomial coefficients and generalized integers
for the special case $\SSS=\ZZ$, for the full set $\sT= \NN$
of allowed bases $b$. 

In Appendix \ref{appendix:tables} we give tables of the first few values of generalized integers, generalized factorials
and generalized binomial coefficients for $(\SSS, \sT)= (\ZZ, \NN)$.

%
%

\subsection{Simultaneous $b$-orderings for $(S,\T)=(\mathbb{Z},\NN)$}
\label{subsec:71}

The first result extends  Bhargava's theorem on simultaneous $p$-ordering for $S= \NN$ 
(Theorem \ref{thm:Bhargava-simultaneous-p-ordering}) to simultaneous $b$-ordering.

%
%

\begin{thm}\label{thm:simultaneous-b-ordering1}
The natural ordering $0,1,2,\dots$ of the nonnegative integers  is a $b$-ordering of $\SSS=\mathbb{Z}$ for all $b\in \NN$
simultaneously.
\end{thm}
%
%

\begin{proof}
For $b=0$ and $b=1$ the $b$-ordering of $\SSS$ is immediate from the proof of Theorem \ref{thm:well-definedness}.
For $b \ge 2$, we have
\begin{equation}\label{eqn:decompose-order_b}
\ordr_b\left(n \right)=\sum_{\substack{i\ge1\\b^i\mid n}}1.
\end{equation}
The proof for $S= \ZZ$  for each $b \ge 2$ proceeds by induction on $k \ge 0$. The base case $k=0$ vacuously holds. 
If  $0,1,2,\dots,k-1$ is a $b$-ordering for the first $k-1$ steps, then at the $k$th step we need to pick $a'= a_k\in\mathbb{Z}$ to minimize
$$
Q:=\sum_{j=0}^{k-1}\ordr_b\left(a'-j\right).
$$
Then, for fixed $a'$, 
\begin{equation}\label{eqn:simultaneous}
Q =\sum_{j=0}^{k-1} \bigg( \sum_{\substack{i\ge 1\\a'\equiv j\bmod{b^i}}} 1 \bigg)
=\sum_{i=1}^{\infty} \bigg( \sum_{\substack{j=0\\j\equiv a'\bmod{b^i}}}^{k-1}1\bigg)
\ge\sum_{i=1}^{\infty} \left\lfloor\frac{k}{b^i}\right\rfloor.
\end{equation}
Since equality holds  if $a'=k$, we may choose $a_k=k$, completing the induction step. 
\end{proof}

%
%

\subsection{Generalized factorials  and generalized integers for  $(S,\T)=(\mathbb{Z},\NN)$}
\label{subsec:72}

%
%
\begin{thm}\label{thm:factorial-factorization}

Let $b \ge 2$ and $k \ge 0$ be integers. Then
\begin{equation}\label{eqn:gamma_k} 
\ivr_k(\ZZ,b) =\sum_{i=1}^\infty\left\lfloor\frac{k}{b^i}\right\rfloor.
\end{equation}
For $k=0,1,2,\dots$, the generalized factorial of $k$ associated to $S=\mathbb{Z}$ and $\T=\NN$  is
\begin{equation}
k!_{\mathbb{Z}, \NN}
=\prod_{b=2}^k b^{\ivr_k(\ZZ,b)}.
\end{equation}
\end{thm}

Theorem \ref{thm:factorial-factorization} parallels Legendre's formula
for the  prime factorization of $k!$, since we have 
$\order_p(k!)=\sum_{i=1}^\infty\left\lfloor\frac{k}{p^i}\right\rfloor= \ivr_k(\ZZ,p)$
for all $p\in\PP$.

%
%
\begin{proof}
The result  directly follows 
   from the proof of Theorem \ref{thm:simultaneous-b-ordering1}, cf. \eqref{eqn:simultaneous}. Note that $\ivr_k(\ZZ, b)=0$ in 
\eqref{eqn:gamma_k}  if $b>k$. For $b=0$ and $b=1$, $\alpha_k(\ZZ, 0) =0$ for all $k \ge 0$, $\alpha_0(\ZZ, 1)=0$
and $\alpha_k(\ZZ,1)= +\infty$ for $k \ge 1$ always contribute values of $1$ to the factorial. 
\end{proof}

We next treat generalized integers.

%
%
\begin{thm}\label{thm:composition}
For $n=1,2,3,\dots$, the $n$th generalized 
integer associated to $S=\mathbb{Z}$ and $\T=\NN$
 is
\begin{equation}
[n]_{\mathbb{Z},\NN}
=\prod_{b=2}^n b^{\ordr_b(n)},
\end{equation}
where $\ordr_b(n)$ is the maximal $\alpha\in\mathbb{N}$ such that $b^\alpha$ divides $n$.
\end{thm}

Theorem \ref{thm:composition} is analogous to the prime factorization of positive integers:
$$
n=\prod_{p\in\PP}p^{\order_p(n)}.
$$

%
%

\begin{proof}
Write $[n]_{\mathbb{Z},\NN}
=\prod_{b=2}^n b^{\delta(n,b)}$. 
From Theorem \ref{thm:factorial-factorization} we have
\[
\delta(n,b) = \ivr_n(\ZZ,b) - \ivr_{n-1}(\ZZ, b)= \sum_{i=1}^{\infty} \left\lfloor\frac{n}{b^i}\right\rfloor - \left\lfloor\frac{n-1}{b^i}\right\rfloor.
\]
The result follows by noting that
$$
\left\lfloor\frac{n}{b^i}\right\rfloor-\left\lfloor\frac{n-1}{b^i}\right\rfloor=
\begin{cases}
1&\text{if }b^i\mid n,\\
0&\text{otherwise.}
\end{cases}
$$
Here \eqref{eqn:decompose-order_b} yields
$$
\ordr_b\left(n\right)=\sum_{i\ge1}\left(\left\lfloor\frac{n}{b^i}\right\rfloor-\left\lfloor\frac{n-1}{b^i}\right\rfloor\right)
$$
giving the result.
\end{proof}

%
%

\subsection{Generalized binomial coefficients for  $(S,\T)=(\mathbb{Z},\NN)$}\label{subsec:73}
We next treat generalized binomial coefficients.
The next result gives a (partial) factorization of the generalized binomial coefficients.

%
%
\begin{thm}\label{thm:binomial-factorization}
Let $k\ge\ell$ be nonnegative integers. Then:
\begin{enumerate}
\item[\emph{(1)}] We have
\begin{equation}
\binom{k}{\ell}_{\mathbb{Z},\NN}
=\prod_{b=2}^k b^{\bt(k,\ell,b)},
\end{equation}
where
\begin{equation}
\bt(k,\ell,b):=\sum_{i=1}^\infty\left(\left\lfloor\frac{k}{b^i}\right\rfloor-\left\lfloor\frac{\ell}{b^i}\right\rfloor-\left\lfloor\frac{k-\ell}{b^i}\right\rfloor\right).
\end{equation}
\item[\emph{(2)}] For $b\ge 2$,
\begin{equation}\label{eqn:beta-digit}
\bt(k,\ell,b)=\frac{1}{b-1}\left(d_b(\ell)+d_b(k-\ell)-d_b(k)\right),
\end{equation}
where $d_b(j)$ is the sum of the base-$b$ digits of $j$.
\end{enumerate}
\end{thm}
%
%

\begin{proof}
(1) follows from the definition and Theorem \ref{thm:factorial-factorization}.

(2) follows from  the identity
$$
d_b(n)=n-(b-1)\sum_{i=1}^\infty\left\lfloor\frac{n}{b^i}\right\rfloor,
$$
substituted term by term into the formula for $\bt(k,\ell,b)$ in (1).
\end{proof}

The formula \eqref{eqn:beta-digit} shows that for all $b$ the exponent $\bt(k,\ell,b)\ge0$,
which implies these generalized binomial coefficients are integers.
This observation is a special case of Theorem \ref{thm:binomial-integer}.

%
%
\begin{rem}\label{rem:74}
 Knuth and Wilf \cite{KW:89} studied  a notion of generalized binomial coefficients $\binom{k}{\ell}_\mathcal{C}$ 
associated to   a sequence $\mathcal{C}=\left(C_n\right)_{n=1}^\infty$ of positive integers (``generalized integers") via the definition
$$
\binom{k}{\ell}_\mathcal{C} :=\prod_{j=1}^\ell\frac{C_{k-j+1}}{C_j}
$$
for  $0\le\ell\le k$. They called such a sequence  $\mathcal{C}$ {\em regularly divisible} if it satisfies the condition
$$
\gcd\left(C_m,C_n\right)=C_{\gcd(m,n)}
$$
for all positive integers $m$ and $n$. They proved that for all regularly divisible sequences
 all associated generalized binomial coefficients are  integers.

In the case of  generalized integers  $[n]_{\mathbb{Z},\sT}$
 it can be shown for any   $\sT \subseteq \PP$ that the
 sequence of  (Bhargava) generalized-integers $\mathcal{C} :=\left([n]_{\mathbb{Z},\sT}\right)$ is a regularly divisible sequence. 
However for    $\sT =\NN$
the  sequence
 $\mathcal{C}_1=\left([n]_{\mathbb{Z},\NN}
 \right)_{n=1}^\infty$  is not regularly divisible, since
$$
\gcd\left([4]_{\mathbb{Z},\NN}, [6]_{\mathbb{Z},\NN} \right)=\gcd(16,36) =4 
\quad\text{but}
\quad[\gcd(4,6)]_{\mathbb{Z},\NN}=[2]_{\mathbb{Z},\NN}=2.
$$
Nevertheless  Theorem \ref{thm:binomial-factorization} shows for $\sT=\NN$  that all 
 $\binom{k}{\ell}_{\mathcal{C}_1} :=\binom{k}{\ell}_{\mathbb{Z},\NN}$
are integers.
\end{rem}

%
%
%
%

\subsection{Generalized products of binomial coefficients for  $(S,\T)=(\mathbb{Z},\NN)$}
\label{subsec:74}

The product of binomial coefficients in the $n$-th row of 
Pascal's triangle itself was previously studied in \cite{LagM:2016} and  \cite{DL:22}.
The integer sequence
\begin{equation}\label{eqn:binom-prod}
\Gb_n= \prod_{\ell=0}^n {{n}\choose{\ell}}
\end{equation}
given by the product of the binomial coefficients given in
the $n$-th row of Pascal's triangle arises as the
reciprocal of the product of all nonzero unreduced Farey fractions $\frac{k}{\ell}$ having
$1 \le k \le \ell \le n$.
The paper \cite[Theorem 5.1]{LagM:2016} computed its prime factorization to be
\begin{equation} \label{eqn:binomial-prod-fact}
\Gb_n= \prod_ {p \le n} p^{\nub(n,p)},
\end{equation}
in which
$$
\nub(n, p) = \frac{2}{p-1}S_p(n) - \frac{n-1}{p-1} d_p(n),
$$
where $d_p(n)$ is the sum of the base $p$ digits of $n$,
and $S_p(n) = \sum_{k=0}^{n-1} d_p(n)$ denotes  their cumulative sum.

We define an analogue, the product of the generalized binomial coefficients
for $(S,\T)=(\mathbb{Z},\NN)$.  This analogue was introduced  in \cite{DLY:22+},  defined 
by the right side of equation \eqref{eqn:gen-binom-prod} below.

%
%
\begin{thm}\label{cor:binomial-product-formula}
Let $\Gbb_n$ denote the product of the generalized binomial coefficients associated to $S=\mathbb{Z}$ 
and $\T=\NN$
in the $n$th row of the generalized Pascal's triangle:
\begin{equation}
\Gbb_n:=\prod_{k=0}^n\binom{n}{k}_{\mathbb{Z},\NN}.
\end{equation}
Then for $n=1,2,3,\dots$,
\begin{equation}\label{eqn:gen-binom-prod}
\Gbb_n=\prod_{b=2}^n b^{\nub(n,b)},
\end{equation}
where
\begin{equation}\label{eqn:floor-formula}
\nub(n,b):=\frac{2}{b-1}S_b(n)-\frac{n-1}{b-1}d_b(n),
\end{equation}
where $d_b(j)$ is the sum of the base $b$ digits of $j$,
and $S_b(n):=\sum_{j=1}^{n-1}d_b(j)$.
\end{thm}
%
%

\begin{proof}
By Theorem \ref{thm:binomial-factorization} we  have
\[
\Gbb_n =\prod_{k=0}^n\binom{n}{k}_{\mathbb{Z},\NN} 
=\prod_{k=0}^n\prod_{b=2}^nb^{\bt(n,k,b)}
=\prod_{b=2}^nb^{\sum_{k=0}^n\bt(n,k,b)}.
\] 
From \eqref{eqn:beta-digit} in Theorem \ref{thm:binomial-factorization},
\[
\sum_{k=0}^n\bt(n,k,b)=\sum_{k=1}^{n-1}\frac{1}{b-1}\left(d_b(k)+d_b(n-k)-d_b(n)\right)
=\nub(n,b),
\] 
as asserted.
\end{proof}

%
%
\section{The case  $\SSS= \PP$}\label{sec:7a} 

We treat the special case $\SSS= \PP$, the primes in $\ZZ$,
and determine the $b$-exponent sequences and generalized factorials, for 
all bases $b \ge 2$. Note that $\varphi(b^n) = b^{n-1} \varphi(b)$ for all $n \ge 1$.


\begin{thm}\label{thm:B-prime}
Let $b\ge2$ and  $k\ge0$ be integers. Then
\begin{equation}\label{eqn:P-k-factorized-def}
\ddnu_k(\mathbb{P},b)=\max\left\{0,\left\lfloor\frac{k-\omega(b)}{\varphi(b)}\right\rfloor
+\left\lfloor\frac{k-\omega(b)}{b\varphi(b)}\right\rfloor+\left\lfloor\frac{k-\omega(b)}{b^2\varphi(b)}\right\rfloor+\cdots\right\},
\end{equation}
where $\varphi(b)$ is Euler's totient function of $b$ and $\omega(b)$ denotes the number of distinct prime divisors
 of $b$. In particular, $\ddnu_k(\mathbb{P},b)=0$ if $\varphi(b)+\omega(b)>k$.
\end{thm}

Theorem \ref{thm:B-prime}  yields a  formula for generalized factorials associated to $S=\PP$: For any set $\sT   \subseteq \NN \backslash \{0\}$, 
\begin{equation}\label{eqn:P-k-factorized}
k!_{\PP, \sT} = \prod_{\substack{b \in \sT\\\varphi(b)+\omega(b)\le k}} 
b^{\left\lfloor \frac{k-\omega(b)}{\varphi(b)}\right\rfloor +\left\lfloor \frac{k-\omega(b)}{b \varphi(b)}\right\rfloor + \left\lfloor \frac{k-\omega(b)}{b^2 \varphi(b)}\right\rfloor+\cdots}.
\end{equation}
For the choice $\sT= \PP$, the result specializes to Bhargava's formula in
Theorem \ref{thm:Bhargava-prime}.

To prove Theorem \ref{thm:B-prime} we analytically derive a lower bound for all $\PP$-test sequences,
 and show this bound  is attained, by construction.
 
%
%
\subsection{A combinatorial lemma}\label{subsec:81}

The following combinatorial inequality will play  a useful role in proving Proposition \ref{prop:84}.


\begin{lem}\label{lem:well-distributed}
Let $k\ge0$ and $m\ge1$ be integers. Then:

\emph{(1)} For any integer partition of $k$, given by 
 $n_1,\dots,n_m\ge 0$ such that $\sum_{i=1}^mn_i=k$, there holds
\begin{equation}\label{eqn:key-ineq}
\sum_{i=1}^m\binom{n_i}{2} \ge \sum_{i=0}^{k-1}\left\lfloor\frac{i}{m}\right\rfloor.
\end{equation}

\emph{(2)} Equality holds in {\em (1)} exactly when $n_i= \left\lfloor \frac{k}{m}\right\rfloor+1$ for $k - m \left\lfloor \frac{k}{m} \right\rfloor$
values of $i$, and $n_i= \left\lfloor \frac{k}{m}\right\rfloor$ for $m -k +m\left\lfloor \frac{k}{m} \right\rfloor$ values of $i$.
\end{lem} 

\begin{proof}
(1) We claim that  for  nonnegative integer $n_i$  with $\sum_{i=1}^m n_i =k$, the   sum 
$ S= \sum_{i=1}^m\binom{n_i}{2} $
is (uniquely) minimized when all $n_i$ take values either $\left\lfloor \frac{k}{m} \right\rfloor$ or  $\left\lfloor \frac{k}{m} \right\rfloor+1$. 
 Up to order, there must then be $B= k-m\left\lfloor \frac{k}{m} \right\rfloor$ that take the value $\left\lfloor \frac{k}{m} \right\rfloor+1$ 
and $A= m - B$ that take value $\left\lfloor \frac{k}{m} \right\rfloor +1$, since this $(A, B)$ is the only nonnegative integer solution to  
$k =  A \left\lfloor \frac{k}{m} \right\rfloor + B\left(  \left\lfloor \frac{k}{m} \right\rfloor +1\right)$, with $A>0$.
To show the claim, take  any $S$, look at the largest  and smallest $n_i$, call them $n_1$ and $n_m$ for convenience.
 (The value $n_m=0$ is allowed.)
If they differ by at least $2$, then replacing $(n_1, n_m)$  with $(n_1-1, n_m+1)$ decreases the sum $S$, because
$$
\left(\binom{n_m}{2}+\binom{n_1}{2}\right)- \left(\binom{n_m-1}{2}+\binom{n_1+1}{2}\right)=(n_m-n_1)-1>0.
$$
Repeating this process will eventually results in a configuration of $m$ values all differing by at most $1$, 
and the process stops, which must be the configuration above.  

 Next we claim 
\begin{equation}\label{eqn:floor-sum} 
\sum_{i=0}^{k-1}\left\lfloor\frac{i}{m}\right\rfloor= \left(k- m \left\lfloor \frac{k}{m} \right\rfloor\right) {{\left\lfloor \frac{k}{m}\right\rfloor+1}\choose{2}}
+\left(m -k+  m \left\lfloor \frac{k}{m} \right\rfloor\right){{\left\lfloor \frac{k}{m}\right\rfloor}\choose{2}}.
\end{equation}
To establish this equality, write  $i=mj+\ell$
with $0 \le \ell < m$ and sum over $j$ for fixed $\ell$. 
(For $j = \left\lfloor \frac{k}{m} \right\rfloor$
the  largest allowed value of $\ell$ is $k-1- m \left\lfloor \frac{k}{m} \right\rfloor$.)
Then the sum
  $S=\sum_{i=0}^{k}\left\lfloor\frac{i}{m}\right\rfloor $ partitions into  $k-m\left\lfloor \frac{k}{m}\right\rfloor+1$ subsums 
 $\sum_{j=0}^{\left\lfloor \frac{k}{m}\right\rfloor} j = {{\left\lfloor \frac{k+1}{m}\right\rfloor}\choose{2}},$
and $m-k+m\left\lfloor \frac{k}{m}\right\rfloor-1$ subsums 
$\sum_{j=0}^{\left\lfloor \frac{k}{m}\right\rfloor-1} j=  {{\left\lfloor \frac{k}{m}\right\rfloor}\choose{2}}.$ So
$$
\sum_{i=0}^k\left\lfloor\frac{i}{m}\right\rfloor=\left(k-m\left\lfloor\frac{k}{m}\right\rfloor+1\right)\binom{\left\lfloor\frac{k}{m}\right\rfloor+1}{2}+\left(m-k+m\left\lfloor\frac{k}{m}\right\rfloor-1\right)\binom{\left\lfloor\frac{k}{m}\right\rfloor}{2}.
$$
Subtracting the term $\left\lfloor\frac{k}{m}\right\rfloor=\binom{\left\lfloor\frac{k}{m}\right\rfloor+1}{2}-\binom{\left\lfloor\frac{k}{m}\right\rfloor}{2}$ from both sides proves \eqref{eqn:floor-sum},
giving the claim.
This proves  (1). 

(2) The attainment of equality follows from  \eqref{eqn:floor-sum}. The uniqueness follows from the proof of (1). 
\end{proof} 

One can rewrite \eqref{eqn:floor-sum} more elegantly as  the floor function identity
\begin{equation}\label{eqn:floor-identity}
\sum_{i=0}^{k-1}\left\lfloor\frac{i}{m}\right\rfloor=k\left\lfloor\frac{k}{m}\right\rfloor-m {{\left\lfloor\frac{k}{m}\right\rfloor}\choose{2}}.
\end{equation}

%
%
\subsection{Lower bound for $\PP$-test sequences}\label{subsec:82}

We prove a lower bound on $\ddnu_k(\mathbb{P},b, \mathbf{a})$ 
for initial parts of a $\PP$-test sequence $\mathbf{a}=(a_i)_{i=0}^k$  given by the right side of \eqref{eqn:P-k-factorized-def}
in Proposition \ref{prop:84} below.


\begin{prop}\label{prop:84} 
Let  $b\ge2$.
Then for each integer $k \ge \omega(b)$ 
and all finite $\PP$-test sequences $\mathbf{a}=(a_j)_{j=0}^k$
 the inequality
\begin{equation}\label{eqn:ineq-sym-key}
\sum_{j=1}^{k}\sum_{i=0}^{j-1}\ordr_b(a_j- a_i)
\ge\sum_{j=\omega(b)}^k\bigg(\sum_{\ell\ge1}\left\lfloor\frac{j-\omega(b)}{b^{\ell-1}\varphi(b)}\right\rfloor\bigg)
\end{equation}
holds.
\end{prop}

\begin{proof}
If $a_i=a_j$ for some $0\le i<j\le k$, then $\ordr_b(a_i-a_j)=+\infty$ and \eqref{eqn:ineq-sym-key} holds.
 It remains to treat the case that  $a_0,\dots,a_k$ are all distinct. 
 Since each $a_j$ is prime, it follows that either $\gcd(a_j,b)=1$ or $a_j$ is a prime divisor of $b$. 
 Denote by $\tilde{k}$ the number of $j\in\{0,\dots,k\}$ such that $\gcd(a_j,b)=1$. Since $a_0,\dots,a_k$ are all distinct, and at most $\omega(b)$ prime divisors of $b$
are in the list, one has 
\begin{equation}\label{eqn:counting-indices}
k+1\le\tilde{k}+\omega(b).
\end{equation}
Fix an integer $\ell\ge1$, and let $0< c_1<\dots<c_{\varphi(b^\ell)}< b^{\ell}$ denote 
the  integers relatively prime to $b^\ell$. Denote by $n_h$ the number of $j\in\{0,\dots,k\}$ such that $a_j\equiv c_h\bmod{b^\ell}$. Then
$$
\tilde{k} =\sum_{h=1}^{\varphi(b^\ell)}n_h.
$$
Note that if $a_j$ is a prime divisor of $b$, then the only $i\in\{0,\dots,k\}$ such that $a_i\equiv a_j\bmod{b^\ell}$ is $i=j$. We
lower bound the number of pairs $(i,j)$ such that $0\le i<j\le k$ and $a_i\equiv a_j\bmod{b^\ell}$ by:
$$
\mathop{\sum\sum}_{\substack{0\le i<j\le k\\a_i\equiv a_j\bmod{b^\ell}}}1=\sum_{h=1}^{\varphi(b^\ell)}\bigg(\mathop{\sum\sum}_{\substack{0\le i<j\le k\\a_i\equiv c_h\bmod{b^\ell}\\a_j\equiv c_h\bmod{b^\ell}}}1\bigg)=\sum_{h=1}^{\varphi(b^\ell)}\binom{n_h}{2}.
$$
Applying  Lemma \ref{lem:well-distributed} (1) with $m=\varphi(b^\ell)=b^{\ell-1}\varphi(b)$ 
 to the sum on the right side, we obtain
\begin{equation}\label{eqn:lemma-bound}
\mathop{\sum\sum}_{\substack{0\le i<j\le k\\a_i\equiv a_j\bmod{b^\ell}}}1 =\sum_{h=1}^{\varphi(b^\ell)}\binom{n_h}{2}
\ge \sum_{i=0}^{\tilde{k}-1}\left\lfloor\frac{i}{m}\right\rfloor.
\end{equation}
From \eqref{eqn:counting-indices} it follows that 
\begin{equation}\label{eqn:counting-bound}
\sum_{i=0}^{\tilde{k}-1}\left\lfloor\frac{i}{m}\right\rfloor\ge\sum_{i=0}^{k-\omega(b)}\left\lfloor\frac{i}{m}\right\rfloor=\sum_{i=\omega(b)}^{k}\left\lfloor\frac{i-\omega(b)}{b^{\ell-1}\varphi(b)}\right\rfloor.
\end{equation}
Combining \eqref{eqn:lemma-bound} and \eqref{eqn:counting-bound}, we obtain
\begin{equation}\label{eqn:key-bound}
\mathop{\sum\sum}_{\substack{0\le i<j\le k\\a_i\equiv a_j\bmod{b^\ell}}}1\ge\sum_{i=\omega(b)}^{k}\left\lfloor\frac{i-\omega(b)}{b^{\ell-1}\varphi(b)}\right\rfloor.
\end{equation}

Now, we treat the left side of \eqref{eqn:ineq-sym-key}. We write $\ordr_b(a_i-a_j)=\sum_{\ell\ge1:b^\ell\mid a_i-a_j}1$ 
and interchange the order of summation, obtaining
$$
\sum_{j=1}^k \sum_{i=0}^{k-1} \ordr_b(a_j-a_i)
=\sum_{j=1}^{k}\sum_{i=0}^{k-1}\bigg(\sum_{\substack{\ell\ge1\\b^\ell\mid a_j-a_i}}1\bigg)
=\sum_{\ell\ge1}\bigg(\mathop{\sum\sum}_{\substack{0\le i<j\le k\\a_i\equiv a_j\bmod{b^\ell}}}1\bigg).
$$
Applying \eqref{eqn:key-bound} to the inner sum on the right side, we obtain
$$
\sum_{j=1}^{k}\sum_{i=0}^{j-1} \ordr_b(a_j-a_i)\ge\sum_{\ell\ge1}\bigg(\sum_{i=\omega(b)}^{k}\left\lfloor\frac{i-\omega(b)}{b^{\ell-1}\varphi(b)}\right\rfloor\bigg).
$$
Interchanging the order of summation on the right side yields \eqref{eqn:ineq-sym-key}.
\end{proof}

%
%
\subsection{Proof of Theorem \ref{thm:B-prime}}\label{subsec:83}

\begin{proof}[ Proof of Theorem \ref{thm:B-prime}]
For each $k \ge 0$ set
\begin{equation}
\ddnu_k(\mathbb{P},b)^{\ast} := \max\left\{0,\left\lfloor\frac{k-\omega(b)}{\varphi(b)}\right\rfloor+\left\lfloor\frac{k-\omega(b)}{b\varphi(b)}\right\rfloor+\left\lfloor\frac{k-\omega(b)}{b^2\varphi(b)}\right\rfloor+\cdots\right\}.
\end{equation}
We are to prove that $\ddnu_k(\mathbb{P},b)=\ddnu_k(\mathbb{P},b)^{\ast}$ holds for all $k \ge 0$.

We establish the result in two steps. The first step is a  construction, for each fixed $e \ge 1$, 
of an initial $\PP$-test sequence $\mathbf{a}_e= (a_k)_{k=0}^{N_e}$    for all fixed 
$0 \le  k \le N_e=  \omega(b) -1 + b^{e-1} \varphi(b^e)$,
and a proof that its $\PP$-test sequence exponent  values are
$$
\ddnu_k(\mathbb{P},b, \mathbf{a}_e) = \ddnu_k(\mathbb{P}, b)^{\ast} 
$$
for $0 \le k \le N_e$.

The second step is to 
show this initial $\PP$-test sequence $\mathbf{a}$ is necessarily an initial part of a $b$-ordering for $\PP$, for this range.
 Proposition \ref{prop:84} is used in this step. 
 Consequently the well-definedness Theorem \ref{thm:well-definedness} applies 
to conclude that
 $$
 \ddnu_k(\mathbb{P}, b) = \ddnu_k(\mathbb{P},b, \mathbf{a}_e) \quad \mbox{for} \quad 0 \le k \le N_e.
 $$
 
 Combining the last two equalities establishes that
 $$
 \ddnu_k(\mathbb{P}, b)= \ddnu_k(\mathbb{P}, b)^{\ast}
 $$
 holds for $0 \le k \le N_{e}$. Since $N_e \to \infty$ as  $e \to \infty$, this completes the proof.

{\bf Step 1}. The construction chooses for $0 \le k < \omega(b)$ the values $a_k=p_k$ to be the $\omega(b)$ distinct primes
dividing $b= \prod_{p_i\mid b} p_i^{f_i}$, ordered in increasing order.
 Then it chooses each $a_{\omega_b -1 + m}$ for $1 \le m \le \varphi(b^e) $
to be a prime in the $m$-th residue class $r_{m} \, (\bmod \, b^e)$ which is relatively prime to $b^e$, with
$0< r_{m} < b^e$, 
enumerated  in increasing order. We may write these 
$1= r_1 < r_2 < r_{\varphi (b^e)} = b^e-1$.
Dirichlet's theorem on primes in arithmetic progressions asserts there are infinitely many primes in
each arithmetic progression $r \, (\bmod \, b^e)$ with $(r, b^e)=1$ (a condition equivalent to $(r, b)=1$) , so we can pick such primes;
 the congruence conditions
$(\bmod \, b^e)$ then make all these primes distinct.  
(In this construction  for different $e$   the sequences $\mathbf{a}_e$ 
may have no values $a_j$   in common  except for  the initial $\omega(b)$  elements.)

With this construction for fixed $e \ge 1$,  the condition, for fixed $\ell$ with $1 \le \ell \le e$, that
$$
a_k- a_j \equiv 0 \, (\bmod\, b^{\ell}), \quad \mbox{with} \quad k>j, 
$$
 can  hold if and only if 
\begin{equation}
k- j \equiv 0 \, (\bmod \, \varphi(b^{\ell})) \quad \mbox{and} \quad j \ge \omega(b).
\end{equation}
Note that $\ordr_b(a_k - a_j)=0$ for $0 \le j < \omega(b)$.
We conclude that, for all $k$ with $0 \le k \le \omega(b) + \varphi(b^e)-1$, there holds 
\begin{eqnarray*}
\ddnu_k(\mathbb{P}, b, \mathbf{a}_e) 
&:=& \sum_{j=0}^{k-1} \ordr_b( a_k - a_j) 
= \begin{cases}
0 & \quad 0 \le j < \omega(b)\\
\sum_{\ell=1}^e \left\lfloor \frac{k- \omega(b)}{ \varphi(b^{\ell})} \right\rfloor & \quad \omega(b) \le j <k.
\end{cases}
\end{eqnarray*}
Since $\varphi(b^\ell) = b^{\ell-1} \varphi(b)$, 
and the right side is $\ddnu_k(\mathbb{P}, b)^{\ast}$, 
we conclude 
\begin{equation}\label{eqn:test-seqence-bound}
\ddnu_k(\mathbb{P}, b, \mathbf{a}_e) = \ddnu_k(\mathbb{P}, b)^{\ast}
\end{equation}
holds for $1\le k \le N_e$.

{\bf Step 2.} We currently know only that $\mathbf{a}_e$ is a $\PP$-test sequence.
We show that $\mathbf{a}_e= (a_j)_{j=0}^{N_e}$ is the initial part of a  $b$-ordering for $\PP$. This is proved by induction on $k \ge 0$,
for $0 \le k\le N_{e}.$
The induction hypothesis assumes the $b$-ordering property holds for all $0 \le j \le k-1$, and we are to prove it holds
for $k$. 
For the base case $0 \le k \le \omega(b)-1$ it holds automatically since all $\ddnu_{k}( \mathbb{P}, b, \mathbf{a}_e)=0$
in those cases.
Now suppose the equality  is true up to  $k-1$, now with $k \ge \omega(b)$.  We use  Proposition \ref{prop:84}, which may
be rewritten as stating for any $\PP$-test sequence $\mathbf{a}^{'}$, and any $k \ge 1$ (with $k \le |\mathbf{a}^{'}|$), 
\begin{equation}\label{eqn:ind-hyp}
\sum_{j=1}^k \ddnu_{j}(\mathbb{P}, b, \mathbf{a}^{'}) \ge \sum_{i= \omega(b) }^k \ddnu_{j}(\mathbb{P}, b)^{\ast}.
\end{equation}
Now consider any $\PP$-test sequence $(a_j^{'})$ that has $a_j= a_j^{'}$ for $0 \le j \le k-1$.
The induction hypothesis gives 
\begin{equation*}
\ddnu_{j}(\mathbb{P}, b, \mathbf{a'})=  \ddnu_j(\mathbb{P}, b, \mathbf{a})=\ddnu_{j}(\mathbb{P}, b)^{\ast}
\end{equation*}
for $0 \le j \le k-1$. Inserting these equalities on the left side of \eqref{eqn:ind-hyp} and canceling equal terms yields,
since $k \ge \omega(b)$,
$$
 \ddnu_{k}(\mathbb{P}, b, \mathbf{a'})\ge  \ddnu_{k}(\mathbb{P}, b)^{\ast}.
$$
But we have by Step 1, 
$$
 \ddnu_{k}(\mathbb{P}, b, \mathbf{a}_e) 
 = \ddnu_{k}(\mathbb{P}, b)^{\ast},
$$
so the choice $a_k^{'}= a_k$ achieves the minimum, showing that  $\mathbf{a}_e$ is a $b$-ordering
up to the $k$-th term, completing the induction step. The induction halts at  $k=N_e$.
\end{proof}

%
%

\section{Concluding remarks}\label{sec:9}

%
%
\subsection{$b$-orderings and generalized  factorials for Dedekind domains}\label{subsec:91}

 Bhargava's theory \cite{Bhar:97a}  constructed generalizations of $\pp$-orderings and $\pp$-sequences
 for arbitrary subsets $\SSS$ of a Dedekind domains $\DD$ (and for Dedekind rings, which are quotients
 of Dedekind domains by a nonzero ideal).   The general $\pp$-sequence  invariants  are ideals of $\DD$, 
   and the parameter $\sT$ should be viewed as running over nonzero (prime) ideals $\pp$ of $\DD$.

  Bhargava used these invariants to define generalized factorials and generalized binomial coefficients as
  fractional  ideals of the ring $\DD$, and  proved  them to be integral ideals.  
Over a number of papers Bhargava showed that  the notion of  $\pp$-orderings and $\pp$-sequences
 have many applications to many problems in commutative algebra, including
 finding bases of rings of integer-valued polynomials on $\SSS$, on which there is a large literature,
 see \cite{CC:97}, \cite{CC:06}, \cite{BharCY:09}.  He also applied them  to give good bases for suitable function spaces.
  In \cite{Bhar:97a}  he used them to construct generalized factorials on Dedekind rings, 
and showed in \cite{Bhar:98} that they 
 determined fixed divisors of primitive polynomials $F(x)$ as the ideal generated by all $F(a)$, 
 evaluated at values $a \in \SSS$.

In 2009 Bhargava \cite{Bhar:09} further  generalized his 
notion of $\pp$-sequence for a set $\SSS$ in  a Dedekind domain $\DD$, introducing 
 refined invariants with extra parameters, termed order-$h$ $\pp$-sequences, and $r$-removed $\pp$-sequences,
attached to sets $\SSS$ of  Dedekind domains, where $h \in \NN \cup \{+\infty\}$ and $r \in \NN$. 
The (additive version) $\pp$-orderings  described  here in Definition \ref{def:b-orderings} correspond to  the choice $r=0$ and $h= +\infty$.
We  will show  elsewhere (\cite{LY:23b})  that the  extension to $b$-orderings given here for $\ZZ$  
can be further extended to arbitrary ideals $\bb$ of  a Dedekind domain $\DD$, 
allowing both  refined  parameters $h$ and $r$. 

 In this paper we treated  generalized factorials for $\DD=\ZZ$
 as  nonnegative integers,  not as ideals, following Bhargava \cite{Bhar:00}. 
The simplification is possible  because the   set  of ideals $\sI(\ZZ)$ of $\ZZ$ is isomorphic to $\NN$,
  viewed as  monoids, with operations  ideal multiplication and  ring multiplication, respectively.
The integer version is compatible with the original definitions of binomial coefficients and factorials. See Edwards \cite{Edw:87}  
 for  history of binomial coefficients.

%
%
\subsection{Partial factorization of products of generalized binomial  coefficients}\label{subsec:92} 

The sequences $\Gb_n$ and $\Gbb_n$ defined in Section \ref{subsec:74}  have some interest
in connection with prime number theory. 
The paper \cite{DL:22}  studied 
 the partial factorizations of $\Gb_n$, defined by
$$
\Gb (n,x) := \prod_{p \le x} p^{\nub(n,p)}.
$$
It determined the asymptotics of $\log \Gb(n, t(n))$ as $n \to \infty$,
where $t(n)= \alpha n$ for fixed $0 < \alpha  \le 1$.
It observed  the existence of a limit scaling function associated to 
the main term of the asymptotics, 
$f(\alpha) = \lim_{n \to \infty} \frac{1}{n^2}\log \Gb(n, \alpha n)$. 
These factorizations relate to prime number theory, in that the remainder term in
the asymptotics of the logarithms of  these partial factorizations $\Gb_n$ 
obtained in  \cite{DL:22} can be  related to the zeta zeros.
This remainder term was shown have a  power savings over the main term, assuming
the Riemann hypothesis. 

 In  \cite{DLY:22+}
Lara Du and   the authors studied partial factorizations  
 defined by the  factorization of $\Gbb(n)$, 
  \begin{equation}\label{eqn:generalized-b-factorization-x}
\Gbb(n, x)  = \prod_{b \le x} b^{\nub(n, b)},
\end{equation}
 and determined the asymptotics
of $\log \Gbb(n, t(n))$ as $n \to \infty$,
where $t(n)= \alpha n$ for fixed $0 < \alpha  \le 1$.
Here there are two limit scaling functions (corresponding to main terms $n^2\log n$ and $n^2$, respectively), and the remainder term is
shown  unconditionally to have  a power savings $O( n^{3/2} )$.

The extension of this paper suggests the  study  of functions of generalized factorials
for $\ZZ$,  in which the $\sT$-parameter is allowed to vary with $n$,  yielding 
 ``partial binomial products" associated to sets $\SSS \subset \ZZ$ and sets of ideals $\sT$.
 This involves   introducing  a (monotone increasing) real function $t =t(n)$ which  allows $\sT$ to vary with $n$,
by  defining 
\begin{equation}
n!_{\SSS, \sT, t} := \prod_{ b \in \sT, b \le t(n)} b^{\ivr_n(\SSS, (b))}. 
\end{equation}

\paragraph{\bf Acknowledgements} 
We thank the reviewer for helpful comments. 
This paper is a  revised  and augmented  version of Chapter 3 of the second  author's
doctoral thesis (\cite{Yangjit:22}),  which was supervised by the first  author.
The first author was partially supported by NSF-grant DMS-1701576. The second author
was partially supported by NSF-grant DMS-1701577.

%
%
\appendix

%
%
\section{Tables of values for
$(S,\T)=(\mathbb{Z},\NN)$}
\label{appendix:tables}


%
%


\subsection{Generalized positive integers for $(S,\T)=(\mathbb{Z},\NN)$}\label{subsection:tables-integers}

The following Table~\ref{tab1} shows the first few values of $[n]_{\mathbb{Z},\NN}$ in decimal form and in factored form.

\bigskip

\begin{table}[H]
{\footnotesize
\begin{center}
\begin{tabular}{|c|rcl|}
\hline
$n$&&$[n]_{\mathbb{Z},\NN}$&\\
\hline\hline
$1$&$1$&$=$&$1$\\
\hline
$2$&$2$&$=$&$2$\\
\hline
$3$&$3$&$=$&$3$\\
\hline
$4$&$16$&$=$&$2^4$\\
\hline
$5$&$5$&$=$&$5$\\
\hline
$6$&$36$&$=$&$2^2\times3^2$\\
\hline
$7$&$7$&$=$&$7$\\
\hline
$8$&$256$&$=$&$2^8$\\
\hline
$9$&$81$&$=$&$3^4$\\
\hline
$10$&$100$&$=$&$2^2\times5^2$\\
\hline
$11$&$11$&$=$&$11$\\
\hline
$12$&$3,456$&$=$&$2^7\times3^3$\\
\hline
$13$&$13$&$=$&$13$\\
\hline
$14$&$196$&$=$&$2^2\times7^2$\\
\hline
$15$&$225$&$=$&$3^2\times5^2$\\
\hline
$16$&$32,768$&$=$&$2^{15}$\\
\hline
$17$&$17$&$=$&$17$\\
\hline
$18$&$17,496$&$=$&$2^3\times3^7$\\
\hline
$19$&$19$&$=$&$19$\\
\hline
$20$&$16,000$&$=$&$2^7\times5^3$\\
\hline
\end{tabular}
\begin{tabular}{|c|rcl|}
\hline
$n$&&$[n]_{\mathbb{Z},\NN}$&\\
\hline\hline
$21$&$441$&$=$&$3^2\times7^2$\\
\hline
$22$&$484$&$=$&$2^2\times11^2$\\
\hline
$23$&$23$&$=$&$23$\\
\hline
$24$&$1,327,104$&$=$&$2^{14}\times3^4$\\
\hline
$25$&$625$&$=$&$5^4$\\
\hline
$26$&$676$&$=$&$2^2\times13^2$\\
\hline
$27$&$6,561$&$=$&$3^8$\\
\hline
$28$&$43,904$&$=$&$2^7\times7^3$\\
\hline
$29$&$29$&$=$&$29$\\
\hline
$30$&$810,000$&$=$&$2^4\times3^4\times5^4$\\
\hline
$31$&$31$&$=$&$31$\\
\hline
$32$&$2,097,152$&$=$&$2^{21}$\\
\hline
$33$&$1,089$&$=$&$3^2\times11^2$\\
\hline
$34$&$1,156$&$=$&$2^2\times17^2$\\
\hline
$35$&$1,225$&$=$&$5^2\times7^2$\\
\hline
$36$&$362,797,056$&$=$&$2^{11}\times3^{11}$\\
\hline
$37$&$37$&$=$&$37$\\
\hline
$38$&$1,444$&$=$&$2^2\times19^2$\\
\hline
$39$&$1,521$&$=$&$3^2\times13^2$\\
\hline
$40$&$10,240,000$&$=$&$2^{14}\times5^4$\\
\hline
\end{tabular}
\begin{tabular}{|c|l|}
\hline
$n$&$[n]_{\mathbb{Z},\NN}$\\
\hline\hline
$41$&$41$\\
\hline
$42$&$2^4\times3^4\times7^4$\\
\hline
$43$&$43$\\
\hline
$44$&$2^7\times11^3$\\
\hline
$45$&$3^7\times5^3$\\
\hline
$46$&$2^2\times23^2$\\
\hline
$47$&$47$\\
\hline
$48$&$2^{25}\times3^5$\\
\hline
$49$&$7^4$\\
\hline
$50$&$2^3\times5^7$\\
\hline
$51$&$3^2\times17^2$\\
\hline
$52$&$2^7\times13^3$\\
\hline
$53$&$53$\\
\hline
$54$&$2^4\times3^{14}$\\
\hline
$55$&$5^2\times11^2$\\
\hline
$56$&$2^{14}\times7^4$\\
\hline
$57$&$3^2\times19^2$\\
\hline
$58$&$2^2\times29^2$\\
\hline
$59$&$59$\\
\hline
$60$&$2^{13}\times3^6\times5^6$\\
\hline
\end{tabular}
\end{center}
}
\caption{$[n]_{\mathbb{Z},\NN}$ in decimal for $1\le n\le40$ and in  factored form for $1\le n\le60$}
\label{tab1}
\end{table}
\bigskip
%
%

\subsection{Generalized factorials for $(S, \sT)= (\ZZ, \NN)$}\label{subsection:tables-factorials}

The following Table~\ref{tab2} shows the first twenty values of $k!_{\mathbb{Z},\NN}$ in decimal form and in factored form.
\bigskip
\begin{table}[H]
{\footnotesize
\begin{center}
\begin{tabular}{|c|rcl|}
\hline
$k$&&$k!_{\mathbb{Z},\NN}$&\\
\hline\hline
$0$&$1$&$=$&$1$\\
\hline
$1$&$1$&$=$&$1$\\
\hline
$2$&$2$&$=$&$2$\\
\hline
$3$&$6$&$=$&$2\times3$\\
\hline
$4$&$96$&$=$&$2^5\times3$\\
\hline
$5$&$480$&$=$&$2^5\times3\times5$\\
\hline
$6$&$17,280$&$=$&$2^7\times3^3\times5$\\
\hline
$7$&$120,960$&$=$&$2^7\times3^3\times5\times7$\\
\hline
$8$&$30,965,760$&$=$&$2^{15}\times3^3\times5\times7$\\
\hline
$9$&$2,508,226,560$&$=$&$2^{15}\times3^7\times5\times7$\\
\hline
$10$&$250,822,656,000$&$=$&$2^{17}\times3^7\times5^3\times7$\\
\hline
$11$&$2,759,049,216,000$&$=$&$2^{17}\times3^7\times5^3\times7\times11$\\
\hline
$12$&$9,535,274,090,496,000$&$=$&$2^{24}\times3^{10}\times5^3\times7\times11$\\
\hline
$13$&$123,958,563,176,448,000$&$=$&$2^{24}\times3^{10}\times5^3\times7\times11\times13$\\
\hline
$14$&$24,295,878,382,583,808,000$&$=$&$2^{26}\times3^{10}\times5^3\times7^3\times11\times13$\\
\hline
$15$&$5,466,572,636,081,356,800,000$&$=$&$2^{26}\times3^{12}\times5^5\times7^3\times11\times13$\\
\hline
$16$&$179,128,652,139,113,899,622,400,000$&$=$&$2^{41}\times3^{12}\times5^5\times7^3\times11\times13$\\
\hline
$17$&$3,045,187,086,364,936,293,580,800,000$&$=$&$2^{41}\times3^{12}\times5^5\times7^3\times11\times13\times17$\\
\hline
$18$&$53,278,593,263,040,925,392,489,676,800,000$&$=$&$2^{44}\times3^{19}\times5^5\times7^3\times11\times13\times17$\\
\hline
$19$&$1,012,293,271,997,777,582,457,303,859,200,000$&$=$&$2^{44}\times3^{19}\times5^5\times7^3\times11\times13\times17\times19$\\
\hline
\end{tabular}
\end{center}
}
\caption{$k!_{\mathbb{Z},\NN}$ in decimal and in factored form  for $0\le k\le19$}
\label{tab2}
\end{table}
\bigskip

%
%

\subsection{Generalized binomial coefficients for $(S,\T)=(\mathbb{Z},\NN)$}\label{subsection:tables-binomial}


The following Tables~\ref{tab3}~and~\ref{tab4} show the first few values of $\binom{k}{\ell}_{\mathbb{Z},\NN}$ in decimal form and in factored form.

\bigskip

\begin{table}[H]
{\footnotesize
\begin{center}
\begin{tabular}{|c|rrrrrrrrrrr|}
\hline
$k\backslash\ell$&$0$&$1$&$2$&$3$&$4$&$5$&$6$&$7$&$8$&$9$&$10$\\
\hline\hline
$0$&$1$&&&&&&&&&&\\
\hline
$1$&$1$&$1$&&&&&&&&&\\
\hline
$2$&$1$&$2$&$1$&&&&&&&&\\
\hline
$3$&$1$&$3$&$3$&$1$&&&&&&&\\
\hline
$4$&$1$&$16$&$24$&$16$&$1$&&&&&&\\
\hline
$5$&$1$&$5$&$40$&$40$&$5$&$1$&&&&&\\
\hline
$6$&$1$&$36$&$90$&$480$&$90$&$36$&$1$&&&&\\
\hline
$7$&$1$&$7$&$126$&$210$&$210$&$126$&$7$&$1$&&&\\
\hline
$8$&$1$&$256$&$896$&$10,752$&$3,360$&$10,752$&$896$&$256$&$1$&&\\
\hline
$9$&$1$&$81$&$10,368$&$24,192$&$54,432$&$54,432$&$24,192$&$10,368$&$81$&$1$&\\
\hline
$10$&$1$&$100$&$4,050$&$345,600$&$151,200$&$1,088,640$&$151,200$&$345,600$&$4,050$&$100$&$1$\\
\hline
\end{tabular}
\end{center}
}
\caption{$\binom{k}{\ell}_{\mathbb{Z},\NN}$ in decimal for $0\le\ell\le k\le10$}
\label{tab3}
\end{table}

\bigskip

\begin{table}[H]
{\footnotesize
\begin{center}
\begin{tabular}{|c|llllllll|}
\hline
$k\backslash\ell$&$0$&$1$&$2$&$3$&$4$&$5$&$6$&$7$\\
\hline\hline
$0$&$1$&&&&&&&\\
\hline
$1$&$1$&$1$&&&&&&\\
\hline
$2$&$1$&$2$&$1$&&&&&\\
\hline
$3$&$1$&$3$&$3$&$1$&&&&\\
\hline
$4$&$1$&$2^4$&$2^3\cdot3$&$2^4$&$1$&&&\\
\hline
$5$&$1$&$5$&$2^3\cdot5$&$2^3\cdot5$&$5$&$1$&&\\
\hline
$6$&$1$&$2^2\cdot3^2$&$2\cdot3^2\cdot5$&$2^5\cdot3\cdot5$&$2\cdot3^2\cdot5$&$2^2\cdot3^2$&$1$&\\
\hline
$7$&$1$&$7$&$2\cdot3^2\cdot7$&$2\cdot3\cdot5\cdot7$&$2\cdot3\cdot5\cdot7$&$2\cdot3^2\cdot7$&$7$&$1$\\
\hline
$8$&$1$&$2^8$&$2^7\cdot7$&$2^9\cdot3\cdot7$&$2^5\cdot3\cdot5\cdot7$&$2^9\cdot3\cdot7$&$2^7\cdot7$&$2^8$\\
\hline
$9$&$1$&$3^4$&$2^7\cdot3^4$&$2^7\cdot3^3\cdot7$&$2^5\cdot3^5\cdot7$&$2^5\cdot3^5\cdot7$&$2^7\cdot3^3\cdot7$&$2^7\cdot3^4$\\
\hline
$10$&$1$&$2^2\cdot5^2$&$2\cdot3^4\cdot5^2$&$2^9\cdot3^3\cdot5^2$&$2^5\cdot3^3\cdot5^2\cdot7$&$2^7\cdot3^5\cdot5\cdot7$&$2^5\cdot3^3\cdot5^2\cdot7$&$2^9\cdot3^3\cdot5^2$\\
\hline
\end{tabular}
\end{center}
}
\caption{$\binom{k}{\ell}_{\mathbb{Z},\NN}$ in factored form for $0\le\ell\le k\le10$, with  $\ell\le7$}
\label{tab4}
\end{table}

\bigskip
%
%

\bigskip
\end{document}